\documentclass[a4paper,10pt]{amsart}
\usepackage[utf8]{inputenc}
 \usepackage{setspace}
\usepackage[cmtip,arrow]{xy}
\usepackage{lscape}
\usepackage{latexsym}
\usepackage[activeacute,english]{babel}
\usepackage{graphicx}
\usepackage{amsmath}
\usepackage{amsfonts}
\usepackage{pst-node}
\usepackage{amsthm}
\usepackage{amsmath}
\usepackage{amssymb}
\usepackage{amscd}
\usepackage{mathrsfs}
\usepackage{pb-diagram,pb-xy}
\xyoption{all}
\usepackage{graphics}
\usepackage{stmaryrd}
\usepackage{yfonts,bbm}
\usepackage{tikz-cd}

\newtheorem{theorem}{Theorem}[section]
\newtheorem{lemma}[theorem]{Lemma}
\newtheorem{proposition}[theorem]{Proposition}
\newtheorem{corollary}[theorem]{Corollary}
\newtheorem{definition}[theorem]{Definition}

\newtheorem{Remark}[theorem]{Remark}

\usepackage[square,sort,comma,numbers]{natbib}
\usepackage{hyperref}
\hypersetup{colorlinks=true,linkcolor=blue,citecolor=red,linktocpage=true}
\usepackage{url}
\usepackage{relsize}

\begin{document}
\title{Characterization of triangular matrix categories via recollements}

\author[M. L. S. Sandoval-Miranda, V. Santiago, E. O. Velasco]{Martha Lizbeth Shaid Sandoval-Miranda,\\ Valente Santiago-Vargas\\ Edgar Omar Velasco-P\'aez}

\thanks{The authors would like to thank the SECIHTI Project CBF-2023-2024-2630, and the Research Support Office of the Metropolitan Autonomous University (DAI UAM) for the support granted.}

\dedicatory{}
\maketitle

\begin{abstract}
In this paper we study  triangular matrix categories using the theory of recollements of abelian categories.  Given a triangular matrix category we construct two  canonical recollements. We show that if certain funtors of these recollements are exact then the category appearing in the middle term  is actually a triangular matrix category. This result is a generalization of one given by  Liping Li in \cite{LipingLi}. We also show that if $\mathrm{Mod}(\mathcal{C})$ admits a nontrivial torsion pair by abelian categories then $\mathcal{C}$ is equivalent to a triangular matrix category.
\end{abstract}

\section{Introduction}
Rings of the form $\left[\begin{smallmatrix}
T & 0 \\ 
M & U
\end{smallmatrix}\right]$ where $T$ and $U$ are rings and $M$ is a $T$-$U$-bimodule have appeared often in the study of the representation theory of artin rings and algebras (see for example \cite{AusPlatRei},\cite{Chase}, \cite{Green2}). Such  rings, called {\it{triangular matrix rings}}, appear naturally in the study of homomorphic images of hereditary  artin algebras and in the study of the decomposition of algebras and the direct sum of two rings. Triangular matrix rings and their homological properties have been widely studied,
(see for example \cite{Happel}, \cite{Haghany}, \cite{Michelena}). 
The so-called one-point extension is a special case of the triangular matrix algebra and this types of algebras have been studied in several contexts. \\
Following Mitchell's philosophy, in \cite{LGOS1} the authors defined the analogous of the triangular matrix algebra to the context of rings with several objects. Given two preadditive categories $\mathcal{U}$ and $\mathcal{T}$ and an additive  functor $M$ from $ \mathcal{U}\otimes \mathcal{T}^{op}$ to the category of abelian groups $\mathbf{Ab}$, $M\in\mathrm{Mod}(\mathcal{U}\otimes \mathcal{T}^{op})$ for short,  they construct the triangular matrix category $\mathbf{\Lambda}:=\left[\begin{smallmatrix}
\mathcal{T} & 0 \\ 
M & \mathcal{U}
\end{smallmatrix}\right]$  and several properties of $\mathrm{Mod}(\mathbf{\Lambda})$, \ the\   category of \ \ additive functors from 
$\mathbf{\Lambda}$ to $\mathbf{Ab}$,   are studied. In \cite{LGOS2} the authors proved that given an ideal of a category $\mathcal{C}$ there exists a canonical recollement. \\
In this paper we continue the study of triangular matrix categories as initiated in \cite{LGOS1} and we proved that a triangular matrix category produces certain type of recollements between abelian categories (see Propositions \ref{primerrecolle} and \ref{segundorecolle}).  In this paper we prove that these types of recollements  characterizes triangular matrix categories (see Theorems \ref{primerequivalencia}  and \ref{segundorsult}). We also prove that  if $\mathrm{Mod}(\mathcal{C})$ has a torsion pair given by non trivial abelian categories hence  $\mathcal{C}$ is equivalent to a triangular matrix category (see Theorem \ref{mismoteorema}).\\
Finally, we include an appendix that recollects certain basic facts that are needed in this paper and which are folklore and well known for experts but we could not find a reference for the proofs of such a results.

\section{Preliminaries}
\subsection{Categorical foundations}
An arbitrary category $\mathcal{C}$ is $\textbf{skeletally}$ $\textbf{small}$ if there is a full subcategory $\mathcal{C}'$ of $\mathcal{C}$ such that the class of objects of $\mathcal{C}'$ is a set and every object of $\mathcal{C}$ is isomorphic to an object in $\mathcal{C}'$.
We recall that a category $\mathcal{C}$ together with an abelian group structure on each of the sets of morphisms $\mathcal{C}(C_{1},C_{2})$ is called  \textbf{preadditive category}  provided all the composition maps
$\mathcal{C}(C,C')\times \mathcal{C}(C',C'')\longrightarrow \mathcal{C}C(C,C'')$
in $ \mathcal{C} $ are bilinear maps of abelian groups. A covariant functor $ F:\mathcal{C}_{1}\longrightarrow \mathcal{C}_{2} $ between  preadditive categories $ \mathcal{C}_{1} $ and $ \mathcal{C}_{2} $ is said to be \textbf{additive} if for each pair of objects $ C $ and $ C' $ in $ \mathcal{C}_{1}$, the map $ F:\mathcal{C}_{1}(C,C')\longrightarrow \mathcal{C}_{2}(F(C),F(C')) $ is a morphism of abelian groups. Let $\mathcal{C}$ and $\mathcal{D}$  be preadditive categories and $\mathbf{Ab}$ the category of abelian groups. A functor $F: \mathcal{C}\times \mathcal{D}\rightarrow \mathbf{Ab}$ is called \textbf{biadditive} if  $F(f+f',g)=F(f,g)+F(f',g)$ $\forall f,f'\in \mathcal{C}(C,C'), \forall g\in \mathcal{D}(D,D')$ and $F(f,g+g')=F(f,g)+F(f,g')$ $\forall f\in \mathcal{C}(C,C'),\forall g,g'\in \mathcal{D}(D,D')$.\\
If $ \mathcal{C} $ is a  preadditive category we always considerer its opposite  category $ \mathcal{C}^{op}$ as a preadditive category by letting $\mathcal{C}^{op} (C',C)= \mathcal{C}(C,C') $. We follow the usual convention of identifying each contravariant  functor $F$ from a category $ \mathcal{C} $ to $ \mathcal{D} $ with the covariant functor $F$ from  $ \mathcal{C}^{op} $ to $ \mathcal{D}$.  An $\textbf{additive}$ $\textbf{category}$ is a preadditive category $\mathcal{C}$ such that every finite family of objects in $\mathcal{C}$ has a coproduct.

\subsection{The category $\mathrm{Mod}(\mathcal{C})$}
Throughout this section $\mathcal{C}$ will be an arbitrary skeletally small preadditive category, and $\mathrm{Mod}(\mathcal{C})$ will denote the \textit{category of covariant functors} from $\mathcal{C}$ to  the category of abelian groups $\mathbf{Ab}$, called the category of $\mathcal{C}$-modules. This category has as objects  the functors from $\mathcal{C}$ to $\mathbf{Ab}$, and a morphism $f:M_{1}\longrightarrow M_{2}$ of $\mathcal{C}$-modules is a natural transformation.  We sometimes we will write for short, $\mathcal{C}(-,\text{?})$
instead of $\mathrm{Hom}_{\mathcal{C}}(-,?)$ and when it is clear from the context we will use just $(-,?).$\\
We now recall some of properties of the category $\mathrm{Mod}(\mathcal{C})$, for more details consult  \cite{AuslanderRep1}. The category $\mathrm{Mod}(\mathcal{C})$ is an abelian with the following properties:
For each $C$ in $\mathcal{C} $, the $\mathcal{C}$-module $(C,-)$ given by $(C,-)(X)=\mathcal{C}(C,X)$ for each $X$ in $\mathcal{C}$, has the property that for each $\mathcal{C}$-module $M$, the map $\left( (C,-),M\right)\longrightarrow M(C)$ given by $f\mapsto f_{C}(1_{C})$ for each $\mathcal{C}$-morphism $f:(C,-)\longrightarrow M$ is an isomorphism of abelian groups. We will often consider this isomorphism an identification.
Hence
\begin{enumerate}
\item The functor $ P:\mathcal{C}\longrightarrow \mathrm{Mod}(\mathcal{C}) $ given by $ P(C)=(C,-) $ is fully faithful.
\item For each family $\lbrace  C_{i}\rbrace _{i\in I}$ of objects in $ \mathcal{C} $, the $ \mathcal{C}$-module $ \underset{i\in I}\amalg P(C_{i}) $ is a projective $ \mathcal{C}$-module.
\item Given a $ \mathcal{C}$-module $ M $, there is a family $ \lbrace C_{i}\rbrace_{i\in I} $ of objects in $ \mathcal{C}$ such that there is an epimorphism $ \underset{i\in I}\amalg P(C_{i})\longrightarrow M\longrightarrow 0 $. We say that $M$ is finitely generated if such family is finite.
\item A $\textbf{finitely generated projective}$ $\mathcal{C}$-module is a direct summand of $ \underset{i\in I}\amalg P(C_{i})$ for some finite family of objects $ \lbrace C_{i}\rbrace_{i\in I} $ in $\mathcal{C}$.
\end{enumerate}

\subsection{Change of Categories}

The results that appears in this subsection  are directly taken from \cite{AuslanderRep1}. Let $ \mathcal{C} $ be a skeletally small preadditive category. There is a unique (up to isomorphism)  functor $ \otimes_{\mathcal{C}}:\mathrm{Mod}(\mathcal{C}^{op})\times \mathrm{Mod}(\mathcal{C})\longrightarrow \mathbf{Ab} $ called the \textbf{tensor product}. The abelian group $\otimes_{\mathcal{C}}(A,B) $  is denoted by $A\otimes_{\mathcal{C}} B $ for all $\mathcal{C}^{op}$-modules $A$ and all $\mathcal{C}$-modules $B$.
\begin{proposition}\label{AProposition0} The tensor product has the following properties:
\begin{enumerate}
\item
\begin{enumerate}
\item For each $ \mathcal{C}$-module $B$, the functor $\otimes_{\mathcal{C}}B:\mathrm{Mod}(\mathcal{C}^{op})\longrightarrow \mathbf{Ab}$ given by $(\otimes_{\mathcal{C}}B)(A)=A\otimes_{\mathcal{C}}B$ for all $\mathcal{C}^{op}$-modules $A$ is right exact.
\item For each $ \mathcal{C}^{op}$-module $A$, the functor $ A\otimes_{\mathcal{C}}:\mathrm{Mod}(\mathcal{C})\longrightarrow \mathbf{Ab}$ given by $(A\otimes_{\mathcal{C}})(B)=A\otimes_{\mathcal{C}}B $ for all $ \mathcal{C}$-modules $B$ is right exact.
\end{enumerate}

\item For each $ \mathcal{C}^{op}$-module $A$ and each $ \mathcal{C}$-module $B$, the functors $A\otimes_{\mathcal{C}}$ and $\otimes_{\mathcal{C}}B $ preserve arbitrary sums.

\item For each object $C$ in $ \mathcal{C} $ we have $ A\otimes_{\mathcal{C}}(C,-) =A(C)$ and $(-,C)\otimes_{\mathcal{C}}B=B(C)$ for all $\mathcal{C}^{op}$-modules $A$ and all $\mathcal{C}$-modules $B$.
\end{enumerate}
\end{proposition}
Suppose now that $ \mathcal{C}' $ is a subcategory of the skeletally small preadditive category $ \mathcal{C} $.
 We use the tensor product of $ \mathcal{C}'$-modules, to describe the left adjoint $ \mathcal{C}\otimes_{\mathcal{C}'} $ of the restriction functor $ \mathrm{res}_{\mathcal{C}'}:\mathrm{Mod}(\mathcal{C})\longrightarrow \mathrm{Mod}(\mathcal{C}')$.\\
Define the functor $ \mathcal{C}\otimes_{\mathcal{C}'}:\mathrm{Mod}\left( \mathcal{C}'\right) \longrightarrow \mathrm{Mod}\left( \mathcal{C}\right)  $ by $ (\mathcal{C}\otimes_{\mathcal{C}'}M)\left( C\right) =\left(-,C\right) \mid_{\mathcal{C}'}\otimes_{\mathcal{C}'} M $ for all $ M\in \mathrm{Mod}\left(\mathcal{C}'\right)  $ and $ C\in\mathcal{C}.$ Using the properties of the tensor product it is not difficult to establish the following proposition.

\begin{proposition}
$\textnormal{\cite[Proposition 3.1]{AuslanderRep1}}$ \label{Rproposition1}
Let $ \mathcal{C}' $ a subcategory of the skeletally small preadditive category $\mathcal{C}$. Then the functor $ \mathcal{C}\otimes_{\mathcal{C}'}:\mathrm{Mod}\left( \mathcal{C}'\right) \longrightarrow \mathrm{Mod}\left( \mathcal{C}\right)  $ satisfies:

\begin{enumerate}
\item $\mathcal{C}\otimes_{\mathcal{C}'}$ is right exact and preserves sums;

\item The composition $ \mathrm{Mod}\left( \mathcal{C}'\right) \overset{\mathcal{C}\otimes_{\mathcal{C}'}}\longrightarrow \mathrm{Mod}\left( \mathcal{C}\right)\overset{\mathrm{res}_{\mathcal{C}'}}\longrightarrow \mathrm{Mod}\left( \mathcal{C}'\right)    $ is the identity on $ \mathrm{Mod}\left( \mathcal{C}'\right); $

\item For each object $ C'\in \mathcal{C}' $, we have $ \mathcal{C}\otimes_{\mathcal{C}'}\mathcal{C}'\left(C',-\right) =\mathcal{C}\left(C', -\right); $

\item For each $\mathcal{C'}$-module $M$ and each $ \mathcal{C} $-module $ N $, the restriction map \[ \mathcal{C}\left( \mathcal{C}\otimes_{\mathcal{C}'}M,N\right)\longrightarrow \mathcal{C}'\left( M,N\mid_{\mathcal{C}'}\right)\] is an isomorphism;
\item $ \mathcal{C}\otimes_{\mathcal{C}'} $ is a fully faithful functor.
\end{enumerate}
\end{proposition}
Having described the left adjoint $ \mathcal{C}\otimes_{\mathcal{C}'} $ of the restriction functor $ \mathrm{res}_{\mathcal{C}'}:\mathrm{Mod}\left(\mathcal{C}\right) \longrightarrow \mathrm{Mod}\left( \mathcal{C}'\right),$ we now describe its right adjoint.\\
Let $ \mathcal{C}' $ be a full subcategory of the category $ \mathcal{C} $. Define the functor $ \mathcal{C}'\left( \mathcal{C},-\right):\mathrm{Mod}\left( \mathcal{C}'\right) \longrightarrow \mathrm{Mod}\left( \mathcal{C}\right)$ by $\mathcal{C}'\left( \mathcal{C},M\right)\left( X\right) =\mathcal{C}'\left( \left(X,-\right) \mid_{\mathcal{C}'},M\right)    $ for all $ \mathcal{C}' $-modules $ M $ and all objects $ X $ in $ \mathcal{C}.$ We have  the following proposition.

\begin{proposition} 
$\textnormal{\cite[Proposition 3.4]{AuslanderRep1}}$ 
\label{Rproposition2}
Let $ \mathcal{C}' $ a subcategory of the skeletally small preadditive category $\mathcal{C}$. Then the functor $ \mathcal{C}'\left( \mathcal{C},-\right) :\mathrm{Mod}\left( \mathcal{C}'\right) \longrightarrow \mathrm{Mod}\left( \mathcal{C}\right)$ has the following properties:
\begin{enumerate}

\item $ \mathcal{C}'\left( \mathcal{C},-\right)  $ is left exact and preserves inverse limits;

\item The composition $\mathrm{Mod}\left( \mathcal{C}'\right) \overset{\mathcal{C}'\left( \mathcal{C},-\right) }\longrightarrow \mathrm{Mod}\left( \mathcal{C}\right)\overset{\mathrm{res}_{\mathcal{C}'}}\longrightarrow \mathrm{Mod}\left( \mathcal{C}'\right)    $ is the identity on $\mathrm{ Mod}\left( \mathcal{C}'\right);$

\item For each $\mathcal{C}' $-module $M $ and $\mathcal{C}$-module $N$, the restriction map 
$$\mathcal{C}\left( N,\mathcal{C}'\left( \mathcal{C},M\right) \right) \longrightarrow \mathcal{C}'\left( N\mid_{\mathcal{C}'},M\right)$$ is an isomorphism;

\item $\mathcal{C}'\left( \mathcal{C},-\right)  $ is a fully faithful functor.
\end{enumerate}
\end{proposition}

\subsection{Varietes and Krull-Schmidt Categories}

Let $\mathcal{C}$ be a preadditive category, we denote by $\mathbf{proj}(\mathrm{Mod}(\mathcal{C}))$ the full subcategory of $\mathrm{Mod}(\mathcal{C})$ consisting of all finitely generated projective $\mathcal{C}$-modules. Let $\mathcal{C}$ be an additive category, it is said that $\mathcal{C}$ is a category in which $\textbf{idempotents split}$ if given $e:C\longrightarrow C$ an idempotent endomorphism of an object $C\in \mathcal{C}$, then $e$ has a kernel in $\mathcal{C}$. It is well known that for a preadditive category $\mathcal{C}$  the category $\mathrm{proj}(\mathrm{Mod}(\mathcal{C}))$ is a skeletally small additive category in which idempotents split, the functor $P:\mathcal{C}\rightarrow \mathrm{proj}(\mathrm{Mod}(\mathcal{C}))$ given by $P(C)=\mathcal{C}(C,-)$, is fully faithful and induces by restriction an equivalence $\mathrm{Mod}(\mathrm{proj}(\mathrm{Mod}(\mathcal{C}))^{op})\simeq \mathrm{Mod}(\mathcal{C})$. We recall the following notion given by Auslander in \cite{AuslanderRep1}.  A $\textbf{variety}$  is a skeletally small, additive category in which idempotents split.\\
Given a ring $R$, we denote by $\mathrm{Mod}(R)$ the category of left $R$-modules and by $\mathrm{mod}(R)$ the full subcategory of $\mathrm{Mod}(R)$ consisting of the finitely generated left $R$-modules.\\
To fix the notation, we recall known results on functors and categories that we use through the paper, referring for the proofs to the papers by Auslander and Reiten \cite{Aus, AuslanderRep1, AusVarietyI}.

\begin{definition}
Let $R$ be a commutative ring.  An $\textbf{R}$$\textbf{-category}$ is a preadditive category such that $\mathcal{C}(C_{1},C_{2})$ is an $R$-module for all pair objects $C_{1}$, $C_{2}$; and the composition is $R$-bilinear. 
An $R$-category $\mathcal{C}$ is $\mathbf{Hom}$-\textbf{finite}, if for each pair of objects $C_{1},C_{2}$ in $\mathcal{C},$ the $R$-module $\mathcal{C}(C_{1},C_{2})$ is finitely generated. An $\textbf{R}$$\textbf{-variety}$ $\mathcal{C}$, is a variety which is an $R$-category. 
\end{definition}

Let $\mathcal{C}$ be an $R$-category. We say that a functor $F:\mathcal{C}\longrightarrow \mathrm{Mod}(R)$ is an $R$-functor  if the induced morphisms $\mathrm{Hom}_{\mathcal{C}}(X,Y)\longrightarrow \mathrm{Hom}_{\mathrm{Mod}(R)}(F(X),F(Y))$ are morphism of $R$-modules for all $X,Y\in \mathcal{C}$. We denote  by  $\mathrm{Fun}_{R}(\mathcal{C},\mathrm{Mod}(R))$ the category whose objects are the  covariant $R$-functors and the morphisms are the natural transformations.\\
Suppose $\mathcal{C}$ is an $R$-category. If $M:\mathcal{C}\longrightarrow \mathbf{Ab}$ is a $\mathcal{C}$-module, then for each $C\in \mathcal{C}$ the abelian group $M(C)$ has a structure of $\mathrm{End}_{\mathcal{C}}(C)^{op}$-module and hence as an $R$-module since $\mathrm{End}_{\mathcal{C}}(C)$ is an $R$-algebra. Further if $f:M\longrightarrow M'$ is a morphism of $\mathcal{C}$-modules it is easy to show that $f_{C}:M(C)\longrightarrow M'(C)$ is a morphism of $R$-modules for each $C\in \mathcal{C}$. Then, is easy to see that there is an isomorphism 
$$\mathrm{Mod}(\mathcal{C})\cong \mathrm{Fun}_{R}(\mathcal{C},\mathrm{Mod}(R)).$$
Let $\mathcal{C}$ be a Hom-finite $R$-variety.  We say that $M\in \mathrm{Mod}(\mathcal{C})$ is \textbf{finitely presented functor} if there exists an exact sequence in $\mathrm{Mod}(\mathcal{C})$
$$\xymatrix{P_{1}\ar[r] & P_{0}\ar[r] & M\ar[r] & 0,}$$
where $P_{1}$ and $P_{0}$ are finitely generated projective $\mathcal{C}$-modules. 
It is easy to see that $P_{0}$ and $P_{1}$  can be chosen as representable functors. Then
a functor $M$ is finitely presented if there exists an exact
sequence
$$\xymatrix{\mathcal{C}(C_1,- )\ar[r] &  \mathcal{C}(C_0,- )\ar[r] &  M\ar[r] & 0,}$$
with $C_{0},C_{1}\in \mathcal{C}$.

\begin{definition}
An additive category $\mathcal{C}$ is \textbf{Krull-Schmidt}, if every
object in $\mathcal{C}$ decomposes in a finite sum of objects whose
endomorphism ring is local.
\end{definition}

\subsection{Tensor Product of Categories}
If $\mathcal{C}$ and $\mathcal{D}$ are $R$-categories, B. Mitchell defined in \cite{Mitchell} the  $\textbf{tensor product}$  $\mathcal{C}\otimes_{R}\mathcal{D} $, whose objects are those of $\mathcal{C}\times \mathcal{D}$ and the morphisms from $(C,D)$ to $(C',D')$ is the ordinary tensor product of $R$-modules $\mathcal{C}(C,C')\otimes_{R}\mathcal{D}(D,D')$. Since that the tensor product is associative and commutative and the composition in $\mathcal{C}$ and $\mathcal{D}$ is $R$-bilinear then the $R$-bilinear composition in $\mathcal{C}\otimes_{R} \mathcal{D}$ is given as follows:
\begin{equation*}
  (f_{2}\otimes g_{2})\circ (f_{1}\otimes g_{1}):=(f_{2}\circ f_{1})\otimes(g_{2}\circ g_{1})
\end{equation*}
for all $f_{1}\otimes g_1\in \mathcal{C}(C,C')\otimes_{R} \mathcal{D}(D,D')$ and  $f_{2}\otimes g_2\in\mathcal{C}(C',C'')\otimes_{R} \mathcal{D}(D',D'')$.
We will sometimes write $\mathcal{C}\otimes \mathcal{D}$ instead of $\mathcal{C}\otimes_{R}\mathcal{D}$.

\begin{definition}
Let $\mathcal{C}$ be an $R$-category.  The $\textbf{enveloping category}$ of $\mathcal{C}$, denoted by $\mathcal{C}^{e}$ is defined as: $\mathcal{C}^{e}:=\mathcal{C}^{op}\otimes_{R}\mathcal{C}.$
\end{definition}

\subsection{Quotient category}

A  $\textbf{two}$ $\textbf{sided}$ $\textbf{ideal}$  $I(-,?)$ is an additive subfunctor of the two variable functor $\mathcal{C}(-,?):\mathcal{C}^{op}\otimes\mathcal{C}\rightarrow\mathbf{Ab}$ such that: (a) if $f\in I(X,Y)$ and $g\in\mathcal C(Y,Z)$, then  $gf\in I(X,Z)$; and (b)
if $f\in I(X,Y)$ and $h\in\mathcal C(U,X)$, then  $fh\in I(U,Z)$. If $I$ is a two-sided ideal, then we can form the $\textbf{quotient category}$  $\mathcal{C}/I$ whose objects are those of $\mathcal{C}$, and where $(\mathcal{C}/I)(X,Y):=\mathcal{C}(X,Y)/I(X,Y)$. Finally the composition is induced by that of $\mathcal{C}$ (see \cite{Mitchell}). There is a canonical projection functor $\pi:\mathcal{C}\rightarrow \mathcal{C}/I$ such that:
\begin{enumerate}

\item $\pi(X)=X$, for all $X\in \mathcal{C}$.

\item For all $f\in \mathcal{C}(X,Y)$, $\pi(f)=f+I(X,Y):=\bar{f}.$
\end{enumerate}





\section{Recollements in functor categories and triangular matrix categories}\label{sec3}
\subsection{Recollements}

We recall some basic definitions. Consider functors $F:\mathcal{C}\rightarrow\mathcal D$ and $G:\mathcal{D}\rightarrow\mathcal{C}$. We say that $F$ is $\textbf{left adjoint}$ to 
$G$ or that $G$ is $\textbf{right adjoint}$ to $F$, and  that $(F,G)$ is an adjoint pair if there is a natural equivalence
$$\eta =\Big\{\eta _{X,Y}:\eta _{X,Y}:\mathrm{Hom}_{\mathcal D}(FX,Y)\rightarrow \mathrm{Hom}_{\mathcal C}(X,GY)\Big\}_{X\in\mathcal C,Y\in\mathcal D }$$
between the functors $\mathrm{Hom}_{\mathcal D}(F(-),-)$ and $\mathrm{Hom}_{\mathcal C}(-,G(-))$.

\begin{definition}\label{recolldef}
Let $\mathcal{A}$, $\mathcal{B}$ and $\mathcal{C}$ be abelian categories. Then the diagram

$$\xymatrix{\mathcal{B}\ar[rr]|{i_{\ast}=i_{!}}  &  &\mathcal{A}\ar[rr]|{j^{!}=j^{\ast}}\ar@<-2ex>[ll]_{i^{\ast}}\ar@<2ex>[ll]^{i^{!}}  &   &\mathcal{C}\ar@<-2ex>[ll]_{j_{!}}\ar@<2ex>[ll]^{j_{\ast}}}$$
is called a $\textbf{recollement}$, if the additive functors $i^{\ast},i_{\ast}=i_{!}, i^{!},j_{!},j^{!}=j^{*}$ and $j_{*}$ satisfy the following conditions:
\begin{itemize}
\item[(R1)] $(i^{*},i_{*}=i_{!},i^{!})$ and $(j_{!},j^{!}=j^{*},j_{*})$ are adjoint  triples, i.e. $(i^{*},i_{*})$, $(i_{!},i^{!})$  $(j_{!},j^{!})$ and $(j^{*},j_{*})$ are adjoint pairs;
\item[(R2)] $\mathrm{Ker}(j^{*})=\mathrm{Im}(i_{*})$;
\item[(R3)] $i_{*}, j_{!},j_{*}$ are full embedding functors.
\end{itemize}
\end{definition}

Let $\mathcal{C}$ be an additive category and $\mathcal{B}$ be a full additive subcategory of $\mathcal{C}$.  Maurice Auslander introduced  in \cite{AuslanderRep1} three functors such  that,  according with Propositions \ref{Rproposition1} and \ref{Rproposition2}, together form an adjoint triple
$(\mathcal{C}\otimes_{\mathcal{B}},\mathrm{res}_{\mathcal{B}},\mathcal{B}(\mathcal{C},-) )$

\begin{equation}\label{opo}
\xymatrix{\mathrm{Mod}(\mathcal{C})\ar[rrr]|{\mathrm{res}_{\mathcal{B}}}  & &   &\mathrm{Mod}(\mathcal{B}).\ar@<-2ex>[lll]_{\mathcal{C}\otimes_{\mathcal{B}}-}\ar@<2ex>[lll]^{\mathcal{B}(\mathcal{C},-)}}
\end{equation}

Let $\mathcal{B}$ a full additive  subcategory of $ \mathcal{C} $. For all pair of objects $ C,C'\in\mathcal{C} $ we denote by $ I_{\mathcal{B}}\left( C,C'\right)$ the abelian subgroup of $ \mathcal{C}\left( C,C'\right)  $ whose elements are morphism $ f:C\longrightarrow C' $ which factor through $ \mathcal{B}.$ It is not hard to see that under these conditions $ I_{\mathcal{B}}\left( -,?\right)  $ is a two sided ideal of $\mathcal{C}\left( -,?\right) $ (here we use that $\mathcal{B}$ have finite coproducts).
Thus we can considerer the quotient category $ \mathcal{C}/I_{\mathcal{B}} $.\\
The canonical functor $ \pi:\mathcal{C}\longrightarrow \mathcal{C}/I_{\mathcal{B}} $ induces an exact functor  by restriction $  \mathrm{res}_{ \mathcal{C}}:\mathrm{Mod} \left( \mathcal{C}/I_{\mathcal{ B}}\right)\longrightarrow \mathrm{Mod} \left( \mathcal{C}\right)   $ defined by $  \mathrm{res}_{ \mathcal{C}}(F):=F\circ\pi $ for all $ F\in \mathrm{Mod} \left( \mathcal{C}/I_{\mathcal{ B}}\right) $. \\

In order to construct a recollement we will recall the following two functors which will be used throught all this paper, for more details see \cite{LGOS2}.

\begin{Remark}\label{2funtores}
We have a functor $ \frac{\mathcal{C}}{I_{\mathcal{B}}}\otimes_{\mathcal{C}}:\mathrm{Mod}(\mathcal{C})\longrightarrow\mathrm{Mod}(\mathcal{C}/I_{\mathcal{B}}) $ defined as follows: 
$\left( \frac{\mathcal{C}}{I_{\mathcal{B}}}\otimes _{\mathcal{C}}M\right)(C):= \frac{\mathcal{C}(-,C)}{I_{\mathcal{B}}(-,C)}\otimes_{\mathcal{C}}M$  for all $M\in \mathrm{Mod}(\mathcal{C})$
and  $\left( \frac{\mathcal{C}}{I_{\mathcal{B}}}\otimes _{\mathcal{C}}M\right)(\overline{f})=\frac{\mathcal{C}}{I_{\mathcal{B}}}(-,f)\otimes_{\mathcal{C}}M$ for all $\overline{f}=f+I_{\mathcal{B}}(C,C')\in \frac{\mathcal{C}(C,C')}{I_{\mathcal{B}}(C,C')}$.\\
We have a functor $\mathcal{C}(\frac{\mathcal{C}}{I_{\mathcal{B}}},-):\mathrm{Mod}\left(\mathcal{C}\right) \longrightarrow \mathrm{Mod}\left(\mathcal{C}/I_{\mathcal{B}}\right) $ given by: 
$\mathcal{C}(\frac{\mathcal{C}}{I_{\mathcal{B}}},M)(C)=\mathcal{C}\left(\frac{\mathcal{C}(C,-)}{I_{\mathcal{B}}(C,-)},M\right)$ $\forall$ $ M\in \mathrm{Mod}(\mathcal{C})$ and  $C\in \mathcal{C}/I_{\mathcal{B}}$ and 
$\mathcal{C}(\frac{\mathcal{C}}{I_{\mathcal{B}}},M)(\overline{f})=\mathcal{C}\left( \frac{\mathcal{C}}{I_{\mathcal{B}}}(f,-),M\right) 
$ for all $\overline{f}=f+I_{\mathcal{B}}(C,C')\in \mathcal{C}(C,C')/I_{\mathcal{B}}(C,C') $ with $C,C'$ in $\mathcal{C}$.
\end{Remark}

The following proposition is well known, see example 3.12 in \cite{Psaro3} and Theorem 3.6 in \cite{LGOS2}.

\begin{proposition}\label{recollementfunctor}
Let $\mathcal{C}$ be an additive category and $\mathcal{B}$ be a full additive subcategory of $\mathcal{C}$. Then there is a recollement:
$$\xymatrix{\mathrm{Mod}(\mathcal{C}/\mathcal{I}_{\mathcal{B}})\ar[rrr]|{\mathrm{res}_{\mathcal{C}}}  &  &&\mathrm{Mod}(\mathcal{C})\ar[rrr]|{\mathrm{res}_{\mathcal{B}}}
\ar@<-2ex>[lll]_{\mathcal{C}/\mathcal{I}_{\mathcal{B}}\otimes_{\mathcal{C}}-}\ar@<2ex>[lll]^{\mathcal{C}(\mathcal{C}/\mathcal{I}_{\mathcal{B}},-)}  & &   &\mathrm{Mod}(\mathcal{B}).\ar@<-2ex>[lll]_{\mathcal{C}\otimes_{\mathcal{B}}-}\ar@<2ex>[lll]^{\mathcal{B}(\mathcal{C},-)}}$$
\end{proposition}

\subsection{Triangular matrix categories}

We consider the triangular matrix category  $\Lambda:=\left[\begin{smallmatrix}
\mathcal{T} & 0 \\ 
M & \mathcal{U}
\end{smallmatrix}\right]$ constructed in the article \cite{LGOS1} and defined as follows.
\begin{definition}$\textnormal{\cite[Definition 3.5]{LGOS1}}$ \label{defitrinagularmat}
Let $\mathcal{U}$ and $\mathcal{T}$ be two $R$-categories and consider an additive $R$-functor  $M$ from the tensor product category  $\mathcal{U}\otimes_{R} \mathcal{T}^{op}$ to the category $\mathrm{Mod}(R)$.
The \textbf{triangular matrix category}
$\Lambda=\left[ \begin{smallmatrix}
\mathcal{T} & 0 \\ M & \mathcal{U}
\end{smallmatrix}\right]$ is defined as follows.
\begin{enumerate}
\item [(a)] The class of objects of this category are matrices $ \left[
\begin{smallmatrix}
T & 0 \\ M & U
\end{smallmatrix}\right]  $ with $ T\in \mathrm{obj}(\mathcal{T}) $ and $ U\in \mathrm{obj}(\mathcal{U}) $.

\item [(b)] For objects in
$\left[ \begin{smallmatrix}
T & 0 \\
M & U
\end{smallmatrix} \right] ,  \left[ \begin{smallmatrix}
T' & 0 \\
M & U'
\end{smallmatrix} \right]$ in
$\Lambda$ we define $\mathrm{ Hom}_{\Lambda}\left (\left[ \begin{smallmatrix}
T & 0 \\
M & U
\end{smallmatrix} \right] ,  \left[ \begin{smallmatrix}
T' & 0 \\
M & U'
\end{smallmatrix} \right]  \right)  := \left[ \begin{smallmatrix}
\mathrm{Hom}_{\mathcal{T}}(T,T') & 0 \\
M(U',T) & \mathrm{Hom}_{\mathcal{U}}(U,U')
\end{smallmatrix} \right].$
\end{enumerate}
The composition is given by
\begin{eqnarray*}
\circ&:&\left[  \begin{smallmatrix}
{\mathcal{T}}(T',T'') & 0 \\
M(U'',T') & {\mathcal{U}}(U',U'')
\end{smallmatrix}  \right] \times \left[
\begin{smallmatrix}
{\mathcal{T}}(T,T') & 0 \\
M(U',T) & {\mathcal{U}}(U,U')
\end{smallmatrix} \right]\longrightarrow\left[
\begin{smallmatrix}
{\mathcal{T}}(T,T'') & 0 \\
M(U'',T) & {\mathcal{U}}(U,U'')\end{smallmatrix} \right] \\
&& \left( \left[ \begin{smallmatrix}
t_{2} & 0 \\
m_{2} & u_{2}
\end{smallmatrix} \right], \left[
\begin{smallmatrix}
t_{1} & 0 \\
m_{1} & u_{1}
\end{smallmatrix} \right]\right)\longmapsto\left[
\begin{smallmatrix}
t_{2}\circ t_{1} & 0 \\
m_{2}\bullet t_{1}+u_{2}\bullet m_{1} & u_{2}\circ u_{1}
\end{smallmatrix} \right].
\end{eqnarray*}
\end{definition}
We recall that $ m_{2}\bullet t_{1}:=M(1_{U''}\otimes t_{1}^{op})(m_{2})$ and
$u_{2}\bullet m_{1}=M(u_{2}\otimes 1_{T})(m_{1})$.
Then, it is clear that $\Lambda $ is an $R$-category since
$\mathcal{T} $ and $\mathcal{U}$ are $R$-categories and $M(U',T)$ is an $R$-module.\\
We define a functor $\Phi:\Lambda\longrightarrow \mathcal{U},$
as follows: $\Phi\Big(\left[\begin{smallmatrix}
 T & 0 \\ 
M & U
\end{smallmatrix}\right]\Big):=U$ and for 
 $\left[\begin{smallmatrix}
\alpha & 0 \\ 
m & \beta
\end{smallmatrix}\right]:\left[\begin{smallmatrix}
 T & 0 \\ 
M & U
\end{smallmatrix}\right]\longrightarrow \left[\begin{smallmatrix}
T' & 0 \\ 
M & U'
\end{smallmatrix}\right]$ we set $\Phi\Big(\left[\begin{smallmatrix}
\alpha & 0 \\ 
m & \beta
\end{smallmatrix}\right]\Big)=\beta$.\\
In this case $\mathcal{U}(-,-)\circ (\Phi^{op}\otimes\Phi)\in\mathrm{Mod}(\Lambda^{e})$ and we can think $\mathcal{U}$ as a $\mathrm{Mod}(\Lambda^{e})$-module. \\
For simplicity we will write,
$\mathfrak{M}=\left[\begin{smallmatrix}
 T & 0 \\ 
M & U
\end{smallmatrix}\right]\in \Lambda$.
Hence, we have morphism in $\mathrm{Mod}(\Lambda^{e})$: $\xymatrix{\Lambda\ar[rr]^(.3){\Gamma(\Phi)} & & \mathcal{U}(-,-)\circ (\Phi^{op}\otimes\Phi)}$ where
$$[\Gamma(\Phi)]_{(\mathfrak{M},\mathfrak{M}')}:\Lambda(\mathfrak{M},\mathfrak{M}')\longrightarrow \mathcal{U}(\Phi(\mathfrak{M}),\Phi(\mathfrak{M}'))$$
is given as $[\Gamma(\Phi)]_{(\mathfrak{M},\mathfrak{M}')}\Big(\left[\begin{smallmatrix}
\alpha & 0 \\ 
m & \beta
\end{smallmatrix}\right]\Big):=\Phi\Big(\left[\begin{smallmatrix}
\alpha & 0 \\ 
m & \beta
\end{smallmatrix}\right]\Big)$ for all $\left[\begin{smallmatrix}
\alpha & 0 \\ 
m & \beta
\end{smallmatrix}\right]\in \Lambda(\mathfrak{M},\mathfrak{M}')$.

\begin{lemma}\label{exacseuqnimpor}
 There exists an exact sequence
in $\mathrm{Mod}(\Lambda^{e})$
$$\xymatrix{0\ar[r] & \mathcal{I}\ar[r] & \Lambda\ar[rr]^(.3){\Gamma(\Phi)} & & \mathcal{U}(-,-)\circ (\Phi^{op}\otimes\Phi)\ar[r] & 0}$$
where for objects $\mathfrak{M}=\left[\begin{smallmatrix}
 T & 0 \\ 
M & U
\end{smallmatrix}\right]$ and $\mathfrak{M}'=\left[\begin{smallmatrix}
 T' & 0 \\ 
M & U'
\end{smallmatrix}\right]$ in $\Lambda$ the ideal $\mathcal{I}$ is given as  $\mathcal{I}\big(\mathfrak{M},\mathfrak{M}'\big)=\mathrm{Ker}\left([\Gamma (\Phi)]_{\big(\mathfrak{M},\mathfrak{M}'\big)}\right)=\left[\begin{smallmatrix}
 \mathcal{T}(T,T') & 0 \\ 
M(U',T) &  0
\end{smallmatrix}\right]$.
\end{lemma}
\begin{proof}
Straightforward.
\end{proof}

\subsection{Torsion pairs}
Let $\mathcal{A}$ be an arbitrary category and $\mathcal{B}$ a full subcategory of $\mathcal{A}$. We define $${}^{\perp_{0}}\mathcal{B}:=\{X\in \mathcal{A}\mid \mathrm{Hom}_{\mathcal{A}}(X,B)=0\,\,\forall B\in \mathcal{B}\}.$$
Dually, we define $\mathcal{B}^{\perp_{0}}$.
\begin{definition}
Let $\mathcal{A}$ be complete and cocomplete abelian category. Let $\mathcal{T}$ and $\mathcal{F}$ subcategories of $\mathcal{A}$. It is said that $(\mathcal{T},\mathcal{F})$ is a \textbf{torsion pair} if
$\mathcal{T}={}^{\perp_{0}}\mathcal{F}$ and $\mathcal{F}=\mathcal{T}^{\perp_{0}}$.
\end{definition}

We refer to chapter VI of  Stenstrom's book \cite{Stentrom}, for basic results on torsion pairs. It is well known that if $(\mathcal{T},\mathcal{F})$ is a torsion pair, hence $\mathcal{T}$ is closed under quotients, extensions and coproducts; and  $\mathcal{F}$ is closed under subobjects, extensions and products.\\
Conversely, if $\mathcal{T}$ is a subcategory of $\mathcal{A}$ which is closed under quotients, extensions and coproducts, there exists subcategory $\mathcal{F}$ such that $(\mathcal{T},\mathcal{F})$ is a torsion pair. Dually, if $\mathcal{F}$ is a subcategory of $\mathcal{A}$ which is closed under subobjects, extensions and products, there exists subcategory $\mathcal{T}$ such that $(\mathcal{T},\mathcal{F})$ is a torsion pair.\\
Now, if $(\mathcal{T},\mathcal{F})$ is a torsion pair of $\mathcal{A}$, there exists a preradical $\tau$ associated to such a torsion pair; and for every $A\in \mathcal{A}$ there exists a canonical exact sequence (which is unique up to isomorphism of exact sequences)
$$\xymatrix{0\ar[r] & \tau(A)\ar[r] & A\ar[r] & \frac{A}{\tau(A)}\ar[r] & 0}$$
where $\tau(A)\in \mathcal{T}$ and $\frac{A}{\tau(A)}\in \mathcal{F}$.

\begin{definition}
Let $\mathcal{A}$ be complete and cocomplete abelian category. Let $(\mathcal{T},\mathcal{F})$  be a torsion pair.
It is said that $(\mathcal{T},\mathcal{F})$ is \textbf{hereditary} if $\mathcal{T}$ is closed under subobjects. It is said that $(\mathcal{T},\mathcal{F})$ is \textbf{cohereditary} if $\mathcal{F}$ is closed under quotients. We say that a triple $(\mathcal{W},\mathcal{T},\mathcal{F})$ is a \textbf{TTF-triple} if $(\mathcal{W},\mathcal{T})$ and $(\mathcal{T},\mathcal{F})$ are torsion pairs.
\end{definition}

We recall the following result.
\begin{proposition}\label{recolleTTF}
Consider a recollement of abelian categories.
$$\xymatrix{\mathcal{B}\ar[rr]|{i_{\ast}=i_{!}}  &  &\mathcal{A}\ar[rr]|{j^{!}=j^{\ast}}\ar@<-2ex>[ll]_{i^{\ast}}\ar@<2ex>[ll]^{i^{!}}  &   &\mathcal{C}\ar@<-2ex>[ll]_{j_{!}}\ar@<2ex>[ll]^{j_{\ast}}}$$
Hence $\Big(\mathrm{Ker}(i^{\ast}), \mathrm{Im}(i_{\ast}),\mathrm{Ker}(i^{!})\Big)$ is a TTF-triple.
\end{proposition}
\begin{proof}
See proof of  \cite[Theorem 4.3]{Psaromodules}.
\end{proof}

\section{Two recollements induced by triangular matrix categories}

We recall the following notions. Let $\mathcal{A}$ be an arbitrary category and $\mathcal{B}$ a full subcategory in $\mathcal{A}$.  The full subcategory $\mathcal{B}$ is $\textbf{contravariantly finite}$ if for every $A\in \mathcal{A}$ there exists a morphism $f_{A}:B\longrightarrow A$ with $B\in \mathcal{B}$ such that if $f':B'\longrightarrow A$ is other morphism with $B'\in \mathcal{B}$, then there exist a morphism $g:B'\longrightarrow B$ such that $f'=f_{A}\circ g$. Dually, is defined the notion of $\textbf{covariantly finite}$. We say that $\mathcal{B}$ is $\textbf{functorially finite}$ if $\mathcal{B}$ is contravariantly finite and covariantly finite.
Consider $\mathcal{U}$ and $\mathcal{T}$  two $R$-categories.

\begin{proposition}\label{basicos}
Let $\mathcal{U}$ and $\mathcal{T}$ be two $R$-categories and  $M\in \mathrm{Mod}(\mathcal{U}\otimes \mathcal{T}^{op})$. Let $\mathcal{T}':=\left[ \begin{smallmatrix}
 \mathcal{T} & 0 \\ 
M &  0
\end{smallmatrix}\right]$ and $\mathcal{U}'=\left[ \begin{smallmatrix}
0 & 0 \\ 
M &  \mathcal{U}
\end{smallmatrix}\right]$ full  additive subcategories of $\Lambda=\left[ \begin{smallmatrix}
 \mathcal{T} & 0 \\ 
M &  \mathcal{U}
\end{smallmatrix}\right]$. Hence
\begin{enumerate}
\item [(a)] $\mathrm{Hom}_{\Lambda}(\mathcal{U}',\mathcal{T}')=0$.
\item [(b)] $\mathrm{Hom}_{\Lambda}\Big(\left[ \begin{smallmatrix}
T & 0 \\ 
M &  0
\end{smallmatrix}\right],\left[ \begin{smallmatrix}
 0 & 0 \\ 
M &  U
\end{smallmatrix}\right]\Big)=M(U,T)$ for $\left[ \begin{smallmatrix}
T & 0 \\ 
M &  0
\end{smallmatrix}\right]\in \mathcal{T'}$ and $\left[ \begin{smallmatrix}
 0 & 0 \\ 
M &  U
\end{smallmatrix}\right]\in \mathcal{U}'$.

\item [(c)] $\mathcal{T}'$ is covariantly finite and $\mathcal{U}'$ is contravariantly finite in $\Lambda$.

\item [(d)] Every object  in $\Lambda$ satisfies that
$$\left[ \begin{smallmatrix}
T & 0 \\ 
M &  U
\end{smallmatrix}\right]=\left[ \begin{smallmatrix}
T & 0 \\ 
M &  0
\end{smallmatrix}\right]\oplus \left[ \begin{smallmatrix}
0 & 0 \\ 
M &  U
\end{smallmatrix}\right]$$

\item [(e)] ${}^{\perp_{0}}(\mathcal{T}')=\mathcal{U}'$ and $(\mathcal{U}')^{\perp_{0}}=\mathcal{T}'.$
\end{enumerate}
\end{proposition}
\begin{proof}
$(a)$. Let $\left[ \begin{smallmatrix}
 0 & 0 \\ 
M &  U
\end{smallmatrix}\right]\in \mathcal{U}'$ and $\left[ \begin{smallmatrix}
 T & 0 \\ 
M &  0
\end{smallmatrix}\right]\in \mathcal{T}'$. Hence $\mathrm{Hom}_{\Lambda}\Big(\left[ \begin{smallmatrix}
 0 & 0 \\ 
M &  U
\end{smallmatrix}\right],\left[ \begin{smallmatrix}
 T & 0 \\ 
M &  0
\end{smallmatrix}\right]\Big)= \left[ \begin{smallmatrix}
\mathrm{Hom}_{\mathcal{T}}(0,T) & 0 \\
M(0,0) & \mathrm{Hom}_{\mathcal{U}}(U,0)
\end{smallmatrix} \right]=\left[ \begin{smallmatrix}
 0 & 0 \\ 
0 &  0
\end{smallmatrix}\right].$ Proving that $\mathrm{Hom}_{\Lambda}(\mathcal{U}',\mathcal{T}')=0$.\\
$(b)$.  $\mathrm{Hom}_{\Lambda}\Big(\left[ \begin{smallmatrix}
T & 0 \\ 
M &  0
\end{smallmatrix}\right],\left[ \begin{smallmatrix}
 0 & 0 \\ 
M &  U
\end{smallmatrix}\right]\Big)=\left[ \begin{smallmatrix}
\mathrm{Hom}_{\mathcal{T}}(T,0) & 0 \\
M(U,T) & \mathrm{Hom}_{\mathcal{U}}(0,U)
\end{smallmatrix} \right]=\left[ \begin{smallmatrix}
0 & 0 \\
M(U,T) & 0
\end{smallmatrix} \right]$.\\
$(c)$ Let  $\left[ \begin{smallmatrix}
T & 0 \\
M & U
\end{smallmatrix} \right]\in \Lambda$. Hence, we have the morphism $\left[ \begin{smallmatrix}
1_{T} & 0 \\
0 & 0
\end{smallmatrix} \right]:\left[ \begin{smallmatrix}
T & 0 \\
M & U
\end{smallmatrix} \right]\longrightarrow \left[ \begin{smallmatrix}
T & 0 \\
M & 0
\end{smallmatrix} \right]$ with $\left[ \begin{smallmatrix}
T & 0 \\
M & 0
\end{smallmatrix} \right] \in \mathcal{T}'$.
Let $\left[ \begin{smallmatrix}
\alpha & 0 \\
0 & 0
\end{smallmatrix} \right]:\left[ \begin{smallmatrix}
T & 0 \\
M & U
\end{smallmatrix} \right]\longrightarrow \left[ \begin{smallmatrix}
T' & 0 \\
M & 0
\end{smallmatrix} \right]$ an arbitrary morphism with $\left[ \begin{smallmatrix}
T' & 0 \\
M & 0
\end{smallmatrix} \right]\in \mathcal{T}'$. Hence, we have the following commutative diagram

$$\xymatrix{{\left[ \begin{smallmatrix}
T & 0 \\
M & U
\end{smallmatrix} \right]}\ar[rr]^{{\left[ \begin{smallmatrix}
\alpha & 0 \\
0 & 0
\end{smallmatrix} \right]}}\ar[rd]_{{\left[ \begin{smallmatrix}
1_{T} & 0 \\
0 & 0
\end{smallmatrix} \right]}} & & {\left[ \begin{smallmatrix}
T' & 0 \\
M & 0
\end{smallmatrix} \right]}\\
 & {\left[ \begin{smallmatrix}
T & 0 \\
M & 0
\end{smallmatrix} \right]}\ar[ur]_{{\left[ \begin{smallmatrix}
\alpha & 0 \\
0 & 0
\end{smallmatrix} \right]}} & }$$
since 
$$\left[ \begin{smallmatrix}
\alpha & 0 \\
0 & 0
\end{smallmatrix} \right]\circ \left[ \begin{smallmatrix}
1_{T} & 0 \\
0 & 0
\end{smallmatrix} \right]=\left[ \begin{smallmatrix}
\alpha & 0 \\
0\bullet 1_{T}+0\bullet 0 & 0
\end{smallmatrix} \right]=\left[ \begin{smallmatrix}
\alpha & 0 \\
0 & 0
\end{smallmatrix} \right]$$
This proves that $\mathcal{T}'$ is covariantly finite.\\
Let  $\left[ \begin{smallmatrix}
T & 0 \\
M & U
\end{smallmatrix} \right]\in \Lambda$. Hence, we have the morphism $\left[ \begin{smallmatrix}
0 & 0 \\
0 & 1_{U}
\end{smallmatrix} \right]:\left[ \begin{smallmatrix}
0 & 0 \\
M & U
\end{smallmatrix} \right]\longrightarrow \left[ \begin{smallmatrix}
T & 0 \\
M & U
\end{smallmatrix} \right]$ with $\left[ \begin{smallmatrix}
0 & 0 \\
M & U
\end{smallmatrix} \right] \in \mathcal{U}'$.
Let $\left[ \begin{smallmatrix}
0 & 0 \\
0 & \alpha
\end{smallmatrix} \right]:\left[ \begin{smallmatrix}
0 & 0 \\
M & U'
\end{smallmatrix} \right]\longrightarrow \left[ \begin{smallmatrix}
T & 0 \\
M & U
\end{smallmatrix} \right]$ an arbitrary morphism with $\left[ \begin{smallmatrix}
0 & 0 \\
M & U'
\end{smallmatrix} \right]\in \mathcal{U}'$. Hence, we have the following commutative diagram

$$\xymatrix{{\left[ \begin{smallmatrix}
0 & 0 \\
M & U'
\end{smallmatrix} \right]}\ar[rr]^{{\left[ \begin{smallmatrix}
0 & 0 \\
0 & \alpha
\end{smallmatrix} \right]}}\ar[rd]_{{\left[ \begin{smallmatrix}
0 & 0 \\
0 & \alpha
\end{smallmatrix} \right]}} & & {\left[ \begin{smallmatrix}
T & 0 \\
M & U
\end{smallmatrix} \right]}\\
 & {\left[ \begin{smallmatrix}
0 & 0 \\
M & U
\end{smallmatrix} \right]}\ar[ur]_{{\left[ \begin{smallmatrix}
0 & 0 \\
0 & 1_{U}
\end{smallmatrix} \right]}} & }$$
since 
$$\left[ \begin{smallmatrix}
0 & 0 \\
0 & 1_{U}
\end{smallmatrix} \right]\circ \left[ \begin{smallmatrix}
0 & 0 \\
0 & \alpha
\end{smallmatrix} \right]=\left[ \begin{smallmatrix}
0& 0 \\
0\bullet 0+1_{U}\bullet 0 & \alpha
\end{smallmatrix} \right]=\left[ \begin{smallmatrix}
0 & 0 \\
0 & \alpha
\end{smallmatrix} \right]$$
Hence $\mathcal{U}'$ is contravariantly finite.\\
$(d)$. We have morphisms 
$$u=\left[ \begin{smallmatrix}
1_{T} & 0 \\
0 & 0
\end{smallmatrix} \right]:\left[ \begin{smallmatrix}
T & 0 \\
M & 0
\end{smallmatrix} \right]\longrightarrow \left[ \begin{smallmatrix}
T & 0 \\
M & U
\end{smallmatrix} \right],\quad  v=\left[ \begin{smallmatrix}
0& 0 \\
0 & 1_{U}
\end{smallmatrix} \right]:\left[ \begin{smallmatrix}
0 & 0 \\
M & U
\end{smallmatrix} \right]\longrightarrow \left[ \begin{smallmatrix}
T & 0 \\
M & U
\end{smallmatrix} \right]$$

$$p=\left[ \begin{smallmatrix}
1_{T} & 0 \\
0 & 0
\end{smallmatrix} \right]:\left[ \begin{smallmatrix}
T & 0 \\
M & U
\end{smallmatrix} \right]\longrightarrow \left[ \begin{smallmatrix}
T & 0 \\
M & 0
\end{smallmatrix} \right],\quad  q=\left[ \begin{smallmatrix}
0 & 0 \\
0 & 1_{U}
\end{smallmatrix} \right]:\left[ \begin{smallmatrix}
T & 0 \\
M & U
\end{smallmatrix} \right]\longrightarrow \left[ \begin{smallmatrix}
0 & 0 \\
M & U
\end{smallmatrix} \right]$$
We note that
$$pu=\left[ \begin{smallmatrix}
1_{T} & 0 \\
0 & 0
\end{smallmatrix} \right]\left[ \begin{smallmatrix}
1_{T} & 0 \\
0 & 0
\end{smallmatrix} \right]=\left[ \begin{smallmatrix}
1_{T} & 0 \\
0 & 0
\end{smallmatrix} \right]=1_{\left[ \begin{smallmatrix}
T & 0 \\
M & 0
\end{smallmatrix} \right]}:\left[ \begin{smallmatrix}
T & 0 \\
M & 0
\end{smallmatrix} \right]\longrightarrow \left[ \begin{smallmatrix}
T& 0 \\
M & 0
\end{smallmatrix} \right],$$

$$qv=\left[ \begin{smallmatrix}
0 & 0 \\
0 & 1_{U}
\end{smallmatrix} \right]\left[ \begin{smallmatrix}
0& 0 \\
0 & 1_{U}
\end{smallmatrix} \right]=\left[ \begin{smallmatrix}
0 & 0 \\
0 & 1_{U}
\end{smallmatrix} \right]=1_{\left[ \begin{smallmatrix}
0 & 0 \\
M & U
\end{smallmatrix} \right]}:\left[ \begin{smallmatrix}
0 & 0 \\
M & U
\end{smallmatrix} \right]\longrightarrow \left[ \begin{smallmatrix}
0 & 0 \\
M & U
\end{smallmatrix} \right].$$
Moreover we have that
$$up+vq=\left[ \begin{smallmatrix}
1_{T} & 0 \\
0 & 0
\end{smallmatrix} \right]\left[ \begin{smallmatrix}
1_{T} & 0 \\
0 & 0
\end{smallmatrix} \right]+\left[ \begin{smallmatrix}
0 & 0 \\
0 & 1_{U}
\end{smallmatrix} \right]\left[ \begin{smallmatrix}
0& 0 \\
0 & 1_{U}
\end{smallmatrix} \right]=\left[ \begin{smallmatrix}
1_{T} & 0 \\
0 & 0
\end{smallmatrix} \right]+\left[ \begin{smallmatrix}
0 & 0 \\
0 & 1_{U}
\end{smallmatrix} \right]=\left[ \begin{smallmatrix}
1_{T} & 0 \\
0 & 1_{U}
\end{smallmatrix} \right]=1_{\left[ \begin{smallmatrix}
T & 0 \\
M & U
\end{smallmatrix} \right]}$$
Therefore,
$$\left[ \begin{smallmatrix}
T & 0 \\ 
M &  U
\end{smallmatrix}\right]=\left[ \begin{smallmatrix}
T & 0 \\ 
M &  0
\end{smallmatrix}\right]\oplus \left[ \begin{smallmatrix}
0 & 0 \\ 
M &  U
\end{smallmatrix}\right]$$
$(e)$. By item (a), we have that $\mathcal{U}'\subseteq {}^{\perp_{0}}(\mathcal{T}')$.
Now, consider $\left[ \begin{smallmatrix}
T & 0 \\ 
M &  U
\end{smallmatrix}\right]\in {}^{\perp_{0}}(\mathcal{T}')$.  Hence

$$\mathrm{Hom}_{\Lambda}\Big(\left[ \begin{smallmatrix}
T & 0 \\ 
M &  U
\end{smallmatrix}\right],\left[ \begin{smallmatrix}
T' & 0 \\ 
M &  0
\end{smallmatrix}\right]\Big)=0$$ for all $\left[ \begin{smallmatrix}
T' & 0 \\ 
M &  0
\end{smallmatrix}\right]\in \mathcal{T}'$. In particular, $\mathrm{Hom}_{\Lambda}\Big(\left[ \begin{smallmatrix}
T & 0 \\ 
M &  U
\end{smallmatrix}\right],\left[ \begin{smallmatrix}
T & 0 \\ 
M &  0
\end{smallmatrix}\right]\Big)=0$, but we have $\left[ \begin{smallmatrix}
1_{T} & 0 \\ 
0 &  0
\end{smallmatrix}\right]\in  \mathrm{Hom}_{\Lambda}\Big(\left[ \begin{smallmatrix}
T & 0 \\ 
M &  U
\end{smallmatrix}\right],\left[ \begin{smallmatrix}
T & 0 \\ 
M &  0
\end{smallmatrix}\right])$. Hence $0=1_{T}$, then $T=0$ and thus
$\left[ \begin{smallmatrix}
T & 0 \\ 
M &  U
\end{smallmatrix}\right]=\left[ \begin{smallmatrix}
0 & 0 \\ 
M &  U
\end{smallmatrix}\right]\in \mathcal{U}'$. This proves that ${}^{\perp_{0}}(\mathcal{T}')\subseteq \mathcal{U}'$. Therefore
$${}^{\perp_{0}}(\mathcal{T}')=\mathcal{U}'.$$
By item (a), we have that $\mathcal{T}'\subseteq (\mathcal{U}')^{\perp_{0}}$.
Now, consider $\left[ \begin{smallmatrix}
T & 0 \\ 
M &  U
\end{smallmatrix}\right]\in (\mathcal{U}')^{\perp_{0}}$. Then
$$\mathrm{Hom}_{\Lambda}\Big(\left[ \begin{smallmatrix}
0 & 0 \\ 
M &  U'
\end{smallmatrix}\right],\left[ \begin{smallmatrix}
T & 0 \\ 
M &  U
\end{smallmatrix}\right]\Big)=0$$ for all $\left[ \begin{smallmatrix}
0 & 0 \\ 
M &  U'
\end{smallmatrix}\right]\in \mathcal{U}'$. In particular, $\mathrm{Hom}_{\Lambda}\Big(\left[ \begin{smallmatrix}
0 & 0 \\ 
M &  U
\end{smallmatrix}\right],\left[ \begin{smallmatrix}
T & 0 \\ 
M &  U
\end{smallmatrix}\right]\Big)=0$, but we have $\left[ \begin{smallmatrix}
0 & 0 \\ 
0 &  1_{U}
\end{smallmatrix}\right]\in  \mathrm{Hom}_{\Lambda}\Big(\left[ \begin{smallmatrix}
0 & 0 \\ 
M &  U
\end{smallmatrix}\right],\left[ \begin{smallmatrix}
T & 0 \\ 
M &  U
\end{smallmatrix}\right]\Big)$. Hence $0=1_{U}$, then $U=0$ and thus $\left[ \begin{smallmatrix}
T & 0 \\ 
M &  U
\end{smallmatrix}\right]=\left[ \begin{smallmatrix}
T & 0 \\ 
M &  0
\end{smallmatrix}\right]\in \mathcal{T}'$
This proves that $(\mathcal{U}')^{\perp_{0}}\subseteq \mathcal{T}'$ and hence
$$(\mathcal{U}')^{\perp_{0}}=\mathcal{T}'.$$
\end{proof}

\begin{proposition}
Let $\Lambda=\left[ \begin{smallmatrix}
 \mathcal{T} & 0 \\ 
M &  \mathcal{U}
\end{smallmatrix}\right]$ be a triangular matrix category with $M\in \mathrm{Mod}(\mathcal{U}\otimes \mathcal{T}^{op})$.
Consider the canonical exact sequence in
$\mathrm{Mod}(\Lambda^{e})$
$$\xymatrix{0\ar[r] & \mathcal{I}\ar[r] & \Lambda\ar[rr]^(.3){\Gamma(\Phi)} & & \mathcal{U}(-,-)\circ (\Phi^{op}\otimes\Phi)\ar[r] & 0}$$
where for objects $\mathfrak{M}'=\left[\begin{smallmatrix}
 T' & 0 \\ 
M & U'
\end{smallmatrix}\right]$ and $\mathfrak{M}=\left[\begin{smallmatrix}
 T & 0 \\ 
M & U
\end{smallmatrix}\right]$ in $\Lambda$ the ideal $\mathcal{I}$ is given as  $\mathcal{I}\big(\mathfrak{M},\mathfrak{M}'\big)=\mathrm{Ker}\left([\Gamma (\Phi)]_{\big(\mathfrak{M},\mathfrak{M}'\big)}\right)=\left[\begin{smallmatrix}
 \mathcal{T}(T,T') & 0 \\ 
M(U',T) &  0
\end{smallmatrix}\right]$.
Hence, $\mathcal{I}=I_{\mathcal{T}'}.$
\end{proposition}
\begin{proof}
Consider arbitrary morphism $\left[\begin{smallmatrix}
\alpha & 0 \\ 
m & 0
\end{smallmatrix}\right]:\left[\begin{smallmatrix}
 T & 0 \\ 
M & U
\end{smallmatrix}\right]\longrightarrow \left[\begin{smallmatrix}
 T' & 0 \\ 
M & U'
\end{smallmatrix}\right]$ in $\mathcal{I}\Big(\left[\begin{smallmatrix}
 T & 0 \\ 
M & U
\end{smallmatrix}\right], \left[\begin{smallmatrix}
 T' & 0 \\ 
M & U'
\end{smallmatrix}\right]\Big)$ with $m\in M(U',T)$. We have the following commutative diagram
$$\xymatrix{{\left[\begin{smallmatrix}
 T & 0 \\ 
M & U
\end{smallmatrix}\right]}\ar[rr]^{{\left[\begin{smallmatrix}
 \alpha & 0 \\ 
m & 0
\end{smallmatrix}\right]}}\ar[dr]_{{\left[\begin{smallmatrix}
1_{T} & 0 \\ 
0 & 0
\end{smallmatrix}\right]}} & & {\left[\begin{smallmatrix}
 T' & 0 \\ 
M & U'
\end{smallmatrix}\right]}\\
&  {\left[\begin{smallmatrix}
 T & 0 \\ 
M & 0
\end{smallmatrix}\right]}\ar[ur]_{{\left[\begin{smallmatrix}
 \alpha & 0 \\ 
m & 0
\end{smallmatrix}\right] }} & }$$
with  $\left[\begin{smallmatrix}
 T & 0 \\ 
M & 0
\end{smallmatrix}\right]\in \mathcal{T}'$,  since $\left[\begin{smallmatrix}
 \alpha & 0 \\ 
m & 0
\end{smallmatrix}\right] \left[\begin{smallmatrix}
 1_{T} & 0 \\ 
0 & 0
\end{smallmatrix}\right]=\left[\begin{smallmatrix}
 \alpha & 0 \\ 
m & 0
\end{smallmatrix}\right]$. This proves that

$$\mathcal{I}\Big(\left[\begin{smallmatrix}
 T & 0 \\ 
M & U
\end{smallmatrix}\right], \left[\begin{smallmatrix}
 T' & 0 \\ 
M & U'
\end{smallmatrix}\right]\Big)\subseteq I_{\mathcal{T}'}\Big(\left[\begin{smallmatrix}
 T & 0 \\ 
M & U
\end{smallmatrix}\right], \left[\begin{smallmatrix}
 T' & 0 \\ 
M & U'
\end{smallmatrix}\right]\Big).$$

Now, let $\left[\begin{smallmatrix}
\alpha & 0 \\ 
m & \beta
\end{smallmatrix}\right]:\left[\begin{smallmatrix}
 T & 0 \\ 
M & U
\end{smallmatrix}\right]\longrightarrow \left[\begin{smallmatrix}
 T' & 0 \\ 
M & U'
\end{smallmatrix}\right]$ an arbitrary morphism with $m\in M(U',T)$, which factors throught and object of $\mathcal{T}'$. Hence we have a commutative diagram
$$\xymatrix{{\left[\begin{smallmatrix}
 T & 0 \\ 
M & U
\end{smallmatrix}\right]}\ar[rr]^{{\left[\begin{smallmatrix}
 \alpha & 0 \\ 
m & \beta
\end{smallmatrix}\right]}}\ar[dr]_{{\left[\begin{smallmatrix}
\alpha_{1} & 0 \\ 
0 & 0
\end{smallmatrix}\right]}} & & {\left[\begin{smallmatrix}
 T' & 0 \\ 
M & U'
\end{smallmatrix}\right]}\\
&  {\left[\begin{smallmatrix}
 T'' & 0 \\ 
M & 0
\end{smallmatrix}\right]}\ar[ur]_{{\left[\begin{smallmatrix}
 \alpha_{2} & 0 \\ 
x & 0
\end{smallmatrix}\right] }} & }$$
with $x\in M(U',T'')$.
Hence $\beta=0\circ 0=0$. Hence $\left[\begin{smallmatrix}
\alpha & 0 \\ 
m & \beta
\end{smallmatrix}\right]=\left[\begin{smallmatrix}
\alpha & 0 \\ 
m & 0
\end{smallmatrix}\right]\in \mathcal{I}\Big(\left[\begin{smallmatrix}
 T & 0 \\ 
M & U
\end{smallmatrix}\right], \left[\begin{smallmatrix}
 T' & 0 \\ 
M & U'
\end{smallmatrix}\right]\Big)$. 

This proves that 
$$\mathcal{I}\Big(\left[\begin{smallmatrix}
 T & 0 \\ 
M & U
\end{smallmatrix}\right], \left[\begin{smallmatrix}
 T' & 0 \\ 
M & U'
\end{smallmatrix}\right]\Big)= I_{\mathcal{T}'}\Big(\left[\begin{smallmatrix}
 T & 0 \\ 
M & U
\end{smallmatrix}\right], \left[\begin{smallmatrix}
 T' & 0 \\ 
M & U'
\end{smallmatrix}\right]\Big).$$
Hence $\mathcal{I}=I_{\mathcal{T}'}$.

\end{proof}

Recall that $\mathcal{U}(-,-)\circ (\Phi^{op}\otimes\Phi)\in\mathrm{Mod}(\Lambda^{e})$ and we can think $\mathcal{U}$ as a $\mathrm{Mod}(\Lambda^{e})$-module. \\

\begin{lemma}\label{Kproy}
Let $H:=\mathcal{U}(-,-)\circ (\Phi^{op}\otimes\Phi)$ be and consider the exact sequence in $\mathrm{Mod}(\Lambda^{e})$:
$$\xymatrix{0\ar[r] & \mathcal{I}\ar[r] & \Lambda\ar[r]^{\Gamma(\Phi)} & H\ar[r] & 0.}$$
Hence,   for $\left[\begin{smallmatrix}
 T & 0 \\ 
M & U
\end{smallmatrix}\right]\in \mathrm{Mod}(\Lambda^{op})$ the following exact sequence splits in $\mathrm{Mod}(\Lambda)$:

$$\xymatrix{0\ar[r] & \mathcal{I}\Big({\left[\begin{smallmatrix}
 T & 0 \\ 
M & U
\end{smallmatrix}\right]},-\Big)\ar[r] & \Lambda\Big({\left[\begin{smallmatrix}
 T & 0 \\ 
M & U
\end{smallmatrix}\right]},-\Big)\ar[r]^{\Psi} & H\Big({\left[\begin{smallmatrix}
 T & 0 \\ 
M & U
\end{smallmatrix}\right]},-\Big)\ar[r] & 0}$$
with $\Psi=[\Gamma (\Phi)]_{\Big(\left[\begin{smallmatrix}
 T & 0 \\ 
M & U
\end{smallmatrix}\right],-\Big)}$.
In particular, $\mathcal{I}\Big({\left[\begin{smallmatrix}
 T & 0 \\ 
M & U
\end{smallmatrix}\right]},-\Big)$  and $H\Big({\left[\begin{smallmatrix}
 T & 0 \\ 
M & U
\end{smallmatrix}\right]},-\Big)$ are projective in $\mathrm{Mod}(\Lambda)$.
\end{lemma}
\begin{proof}
We define $\Theta: H\Big({\left[\begin{smallmatrix}
 T & 0 \\ 
M & U
\end{smallmatrix}\right]},-\Big)\longrightarrow \Lambda\Big({\left[\begin{smallmatrix}
 T & 0 \\ 
M & U
\end{smallmatrix}\right]},-\Big)$ as follows: for $\left[\begin{smallmatrix}
 T_{1} & 0 \\ 
M  & U_{1}
\end{smallmatrix}\right]\in \Lambda$ we set
$$\Theta_{\left[\begin{smallmatrix}
 T_{1} & 0 \\ 
M  & U_{1}
\end{smallmatrix}\right]}: H\Big({\left[\begin{smallmatrix}
 T & 0 \\ 
M & U
\end{smallmatrix}\right]}, \left[\begin{smallmatrix}
 T_{1} & 0 \\ 
M  & U_{1}
\end{smallmatrix}\right]\Big)=\mathcal{U}(U,U_{1})\longrightarrow \Lambda\Big({\left[\begin{smallmatrix}
 T & 0 \\ 
M & U
\end{smallmatrix}\right]},\left[\begin{smallmatrix}
 T_{1} & 0 \\ 
M  & U_{1}
\end{smallmatrix}\right]\Big)$$
given as  $\Theta_{\left[\begin{smallmatrix}
 T_{1} & 0 \\ 
M  & U_{1}
\end{smallmatrix}\right]}(\beta)=\left[\begin{smallmatrix}
 0 & 0 \\ 
0  & \beta 
\end{smallmatrix}\right]$  $\forall \beta\in \mathcal{U}(U,U_{1})$. It is easy to see that $\Psi\circ\Theta=1$. 
\end{proof}

Now, we are ready to prove the first recollement induced by a triangular matrix category.

\begin{proposition}\label{primerrecolle}
Let $\Lambda=\left[ \begin{smallmatrix}
 \mathcal{T} & 0 \\ 
M &  \mathcal{U}
\end{smallmatrix}\right]$ be a triangular matrix category with $M\in \mathrm{Mod}(\mathcal{U}\otimes \mathcal{T}^{op})$. Then there exists a recollement
$$\xymatrix{\mathrm{Mod}(\mathcal{U})=\mathrm{Mod}(\Lambda/I_{\mathcal{T}'})\ar[rr]|(.65){\mathrm{res}_{\Lambda}}  &  & \mathrm{Mod}(\Lambda)\ar[rr]|(.4){\mathrm{res}_{\mathcal{T}'}}\ar@<-2ex>[ll]_(.35){\Lambda/I_{\mathcal{T}'}\otimes_{\Lambda}-}\ar@<2ex>[ll]^(.35){\Lambda(\Lambda/I_{\mathcal{T}'},-)} &   & \mathrm{Mod}(\mathcal{T}')=\mathrm{Mod}(\mathcal{T})\ar@<-2ex>[ll]_(.6){\Lambda\otimes_{\mathcal{T}'}}\ar@<2ex>[ll]^(.6){\mathcal{T}'(\Lambda,-)}}$$
where $\Lambda(\Lambda/I_{\mathcal{T}'},-)$ is an exact functor.
\end{proposition}
\begin{proof}
By Proposition \ref{recollementfunctor}, we have the recollement

$$\xymatrix{\mathrm{Mod}(\Lambda/I_{\mathcal{T}'})\ar[rr]|{\mathrm{res}_{\Lambda}}  &  & \mathrm{Mod}(\Lambda)\ar[rr]|{\mathrm{res}_{\mathcal{T}'}}\ar@<-2ex>[ll]_{\Lambda/I_{\mathcal{T}'}\otimes_{\Lambda}-}\ar@<2ex>[ll]^{\Lambda(\Lambda/I_{\mathcal{T}'},-)} &   & \mathrm{Mod}(\mathcal{T}')\ar@<-2ex>[ll]_{\Lambda\otimes_{\mathcal{T}'}}\ar@<2ex>[ll]^{\mathcal{T}'(\Lambda,-)}}$$
We note that for $\mathfrak{M}=\left[\begin{smallmatrix}
 T & 0 \\ 
M & U
\end{smallmatrix}\right]\in \Lambda$, we have that: 
$$H\Big(\mathfrak{M},-\Big)=\mathrm{Hom}_{\Lambda/\mathcal{I}}(\mathfrak{M},-)\simeq \frac{\Lambda(\mathfrak{M},-)}{\mathcal{I}(\mathfrak{M},-)}=\frac{\Lambda(\mathfrak{M},-)}{I_{\mathcal{T}'}(\mathfrak{M},-)}.$$
By Lemma \ref{Kproy}, we conclude that 
$\frac{\Lambda(\mathfrak{M},-)}{I_{\mathcal{T}'}(\mathfrak{M},-)}$ is a projective $\Lambda$-modulo. Now, by Remark \ref{2funtores}, for $M\in \mathrm{Mod}(\Lambda)$ the functor $\Lambda(\Lambda/I_{\mathcal{T}'},M)\in \mathrm{Mod}(\Lambda/I_{\mathcal{T}'})$ is given as follows: for $\mathfrak{M}=\left[\begin{smallmatrix}
 T & 0 \\ 
M & U
\end{smallmatrix}\right]\in \Lambda/I_{\mathcal{T}'}$ we have that
$$\Lambda(\Lambda/I_{\mathcal{T}'},M)(\mathfrak{M})=\mathrm{Hom}_{\mathrm{Mod}(\Lambda)}\Bigg(\frac{\Lambda(\mathfrak{M},-)}{I_{\mathcal{T}'}(\mathfrak{M},-)},M\Bigg)$$
Then, $\Lambda(\Lambda/I_{\mathcal{T}'},-)$ is an exact functor since $\frac{\Lambda(\mathfrak{M},-)}{I_{\mathcal{T}'}(\mathfrak{M},-)}$ is a projective $\Lambda$-modulo. Finally, note that $\Lambda /I_{\mathcal{T}'}=\Lambda/\mathcal{I}\simeq \mathcal{U}'\simeq \mathcal{U}$  and $\mathcal{T}'\simeq \mathcal{T}$. Then we have the required recollement. 
\end{proof}



In order to construct another recollement associated to a triangular matrix category, we  define a functor $\Phi':\Lambda\longrightarrow \mathcal{T},$
as follows: $\Phi'\Big(\left[\begin{smallmatrix}
 T & 0 \\ 
M & U
\end{smallmatrix}\right]\Big):=T$ and for 
 $\left[\begin{smallmatrix}
\alpha & 0 \\ 
m & \beta
\end{smallmatrix}\right]:\left[\begin{smallmatrix}
 T & 0 \\ 
M & U
\end{smallmatrix}\right]\longrightarrow \left[\begin{smallmatrix}
T' & 0 \\ 
M & U'
\end{smallmatrix}\right]$ we set $\Phi'\Big(\left[\begin{smallmatrix}
\alpha & 0 \\ 
m & \beta
\end{smallmatrix}\right]\Big)=\alpha$.\\
In this case $\mathcal{T}(-,-)\circ ((\Phi')^{op}\otimes\Phi')\in\mathrm{Mod}(\Lambda^{e})$ and we can think $\mathcal{T}$ as a $\mathrm{Mod}(\Lambda^{e})$-module. \\
Recall, that for simplicity we write,
$\mathfrak{M}=\left[\begin{smallmatrix}
 T & 0 \\ 
M & U
\end{smallmatrix}\right]\in \Lambda$.
Hence, we have morphism in $\mathrm{Mod}(\Lambda^{e})$: $\xymatrix{\Lambda\ar[rr]^(.3){\Gamma(\Phi')} & & \mathcal{T}(-,-)\circ ((\Phi')^{op}\otimes\Phi')}$ where
$$[\Gamma(\Phi')]_{(\mathfrak{M},\mathfrak{M}')}:\Lambda(\mathfrak{M},\mathfrak{M}')\longrightarrow \mathcal{T}(\Phi(\mathfrak{M}),\Phi(\mathfrak{M}'))$$
is given as $[\Gamma(\Phi')]_{(\mathfrak{M},\mathfrak{M}')}\Big(\left[\begin{smallmatrix}
\alpha & 0 \\ 
m & \beta
\end{smallmatrix}\right]\Big):=\Phi'\Big(\left[\begin{smallmatrix}
\alpha & 0 \\ 
m & \beta
\end{smallmatrix}\right]\Big)$ for all $\left[\begin{smallmatrix}
\alpha & 0 \\ 
m & \beta
\end{smallmatrix}\right]\in \Lambda(\mathfrak{M},\mathfrak{M}')$.

Hence, we have the following result.
\begin{lemma}\label{Kproy2}
 There exists an exact sequence
in $\mathrm{Mod}(\Lambda^{e})$
$$\xymatrix{0\ar[r] & \mathcal{J}\ar[r] & \Lambda\ar[rr]^(.3){\Gamma(\Phi')} & & \mathcal{T}(-,-)\circ ((\Phi')^{op}\otimes\Phi')\ar[r] & 0}$$
where for objects $\mathfrak{M}'=\left[\begin{smallmatrix}
 T' & 0 \\ 
M & U'
\end{smallmatrix}\right]$ and $\mathfrak{M}=\left[\begin{smallmatrix}
 T & 0 \\ 
M & U
\end{smallmatrix}\right]$ in $\Lambda$ the ideal $\mathcal{J}$ is given as  $\mathcal{J}\big(\mathfrak{M},\mathfrak{M}'\big)=\mathrm{Ker}\left([\Gamma (\Phi')]_{\big(\mathfrak{M},\mathfrak{M}'\big)}\right)=\left[\begin{smallmatrix}
0 & 0 \\ 
M(U',T) &   \mathcal{U}(U,U') 
\end{smallmatrix}\right]$. Moreover, we have that 
\begin{enumerate}
\item [(a)] $\mathcal{J}=I_{\mathcal{U}'}$

\item [(b)] Let $L:=\mathcal{T}(-,-)\circ ((\Phi')^{op}\otimes\Phi')$ be. The following exact sequence splits in $\mathrm{Mod}(\Lambda^{op})$ for every $\mathfrak{M}\in \Lambda$:
$$\xymatrix{0\ar[r] & \mathcal{J}(-,\mathfrak{M})\ar[r] & \Lambda(-,\mathfrak{M})\ar[rr] & & L(-,\mathfrak{M})\ar[r] & 0}.$$
In particular, $\mathcal{J}(-,\mathfrak{M})$ and $L(-,\mathfrak{M})$ are $\Lambda^{op}$-projective modules.
\end{enumerate}
\end{lemma}
\begin{proof}
Straightforward.
\end{proof}

For the following definition see p. 48 in \cite{Mitchell}. 

\begin{definition}
Let $\mathcal{C}$ a preadditive category.  We say that $M\in \mathrm{Mod}(\mathcal{C}^{op})$ is $\textbf{flat}$ if
$M\otimes_{\mathcal{C}}-:\mathrm{Mod}(\mathcal{C})\longrightarrow \text{Ab}$ is an exact functor.
\end{definition}

\begin{proposition}\label{segundorecolle}
Let $\Lambda=\left[ \begin{smallmatrix}
 \mathcal{T} & 0 \\ 
M &  \mathcal{U}
\end{smallmatrix}\right]$ be a triangular matrix category with $M\in \mathrm{Mod}(\mathcal{U}\otimes \mathcal{T}^{op})$.
We have recollement
$$\xymatrix{\mathrm{Mod}(\mathcal{T})=\mathrm{Mod}(\Lambda/I_{\mathcal{U}'})\ar[rr]|(.65){\mathrm{res}_{\Lambda}}  &  & \mathrm{Mod}(\Lambda)\ar[rr]|(.4){\mathrm{res}_{\mathcal{U}'}}\ar@<-2ex>[ll]_(.35){\Lambda/I_{\mathcal{U}'}\otimes_{\Lambda}-}\ar@<2ex>[ll]^(.35){\Lambda(\Lambda/I_{\mathcal{U}'},-)} &   & \mathrm{Mod}(\mathcal{U}')=\mathrm{Mod}(\mathcal{U})\ar@<-2ex>[ll]_(.6){\Lambda\otimes_{\mathcal{U}'}}\ar@<2ex>[ll]^(.6){\mathcal{U}'(\Lambda,-)}}$$

where $\Lambda/I_{\mathcal{U}'}\otimes_{\Lambda}-$ is an exact functor.

\end{proposition}
\begin{proof}
By Proposition \ref{recollementfunctor}, we have recollement
$$\xymatrix{\mathrm{Mod}(\Lambda/I_{\mathcal{U}'})\ar[rr]|{\mathrm{res}_{\Lambda}}  &  & \mathrm{Mod}(\Lambda)\ar[rr]|{\mathrm{res}_{\mathcal{U}'}}\ar@<-2ex>[ll]_{\Lambda/I_{\mathcal{U}'}\otimes_{\Lambda}-}\ar@<2ex>[ll]^{\Lambda(\Lambda/I_{\mathcal{U}'},-)} &   & \mathrm{Mod}(\mathcal{U}')\ar@<-2ex>[ll]_{\Lambda\otimes_{\mathcal{U}'}}\ar@<2ex>[ll]^{\mathcal{U}'(\Lambda,-)}}$$
We note that for $\mathfrak{M}=\left[\begin{smallmatrix}
 T & 0 \\ 
M & U
\end{smallmatrix}\right]\in \Lambda$, we have that: 
$$L\Big(-,\mathfrak{M}\Big)=\mathrm{Hom}_{\Lambda/\mathcal{J}}(-,\mathfrak{M})\simeq \frac{\Lambda(-,\mathfrak{M})}{\mathcal{J}(-,\mathfrak{M})}=\frac{\Lambda(-,\mathfrak{M})}{I_{\mathcal{U}'}(-,\mathfrak{M})}.$$
By Lemma \ref{Kproy2}, we conclude that 
$\frac{\Lambda(-,\mathfrak{M})}{I_{\mathcal{U}'}(-,\mathfrak{M})}$ is a finitely generated projective $\Lambda^{op}$-modulo. Hence $\frac{\Lambda(-,\mathfrak{M})}{I_{\mathcal{U}'}(-,\mathfrak{M})}$  is a  flat $\Lambda^{op}$-modulo (see \cite[Corollary 10.2]{Mitchell}).\\
Now, by Remark \ref{2funtores}, for $M\in \mathrm{Mod}(\Lambda)$ the functor $\Lambda/I_{\mathcal{U}'}\otimes_{\Lambda}M\in \mathrm{Mod}(\Lambda/I_{\mathcal{U}'})$ is given as follows: for $\mathfrak{M}=\left[\begin{smallmatrix}
 T & 0 \\ 
M & U
\end{smallmatrix}\right]\in \Lambda/I_{\mathcal{U}'}$ we have that
$$(\Lambda/I_{\mathcal{U}'}\otimes_{\Lambda}M)(\mathfrak{M})=\frac{\Lambda(-,\mathfrak{M})}{I_{\mathcal{U}'}(-,\mathfrak{M})}\otimes_{\Lambda}M$$
Then, $\Lambda/I_{\mathcal{U}'}\otimes_{\Lambda}-$ is an exact functor since $\frac{\Lambda(-,\mathfrak{M})}{I_{\mathcal{U}'}(-,\mathfrak{M})}$ is a  flat $\Lambda^{op}$-modulo.\\
Finally, we have that $\Lambda /I_{\mathcal{U}'}=\Lambda/\mathcal{J}\simeq \mathcal{T}'\simeq \mathcal{T}$ and $\mathcal{U}'\simeq \mathcal{U}$, then we have the required recollement.
\end{proof}

\section{Characterizing triangular matrix categories}\label{sec:1}

In all this section, $\mathcal{C}$ will be a Krull-Schmidt additive category with splitiing idempotents and $\mathcal{B}$ an additive subcategory of $\mathcal{C}$ such that $\mathcal{B}=\mathrm{add}(\mathcal{B})$. Consider the ideal $\mathcal{I}_{\mathcal{B}}$ the ideal of morphism which factor through objects in $\mathcal{B}$.


We recall the following well-known result

\begin{lemma}\label{splitidem}
Let $\mathcal{C}$ be a  category in which the idempotents split. Let $\theta':C\longrightarrow C$ be an idempotent. Then
$$C=\mathrm{Ker}(1-\theta')\oplus \mathrm{Ker}(\theta').$$
\end{lemma}
\begin{proof}
Consider $\mu':C'\longrightarrow C$ the kernel of $1-\theta'$. Since $(1-\theta')\theta'=0$, there exists $p':C\longrightarrow C'$ such that the following diagram commutes
$$\xymatrix{0\ar[r] & C'\ar[r]^{\mu'} & C\ar[r]^{1-\theta'} & C\\
& C\ar[ur]_{\theta'}\ar[u]^{p'} & }$$
Similarly, let $\mu'':C''\longrightarrow C$ the kernel of $\theta'$. Since $\theta'(1-\theta')=0$, there exists $p'':C\longrightarrow C''$ such that the following diagram commutes
$$\xymatrix{0\ar[r] & C''\ar[r]^{\mu''} & C\ar[r]^{\theta'} & C\\
& C\ar[ur]_{1-\theta'}\ar[u]^{p''} & }$$
It is easy to see that $C\simeq C'\oplus C''$ where $\mu':C'\longrightarrow C$  and $\mu'':C''\longrightarrow C$ are the inclusions, $p':C\longrightarrow C'$  $p'':C\longrightarrow C''$  are the projections.
\end{proof}

We shall use the following definition.

\begin{definition}\label{defisplitting}
Let $\mathcal{C}$ be a Krull-Schmidt category and $(\mathcal{U}, \mathcal{T})$ a pair of additive full subcategories of $\mathcal{C}$. It is said that $(\mathcal{U}, \mathcal{T})$ is a $\textbf{splitting torsion pair}$  if 
\begin{itemize}
\item[(i)] For all $X\in \mathrm{ind}(\mathcal{C})$, then either $X\in \mathcal{U}$ or $X\in \mathcal{T}$.
\item[(ii)] $\mathrm{Hom}_{\mathcal{C}}(X,-)\mid_{\mathcal{T}}=0$ for all $X\in\mathcal{U}$.
\end{itemize}
\end{definition}

\begin{Remark}
If  $(\mathcal{Y},\mathcal{X})$ is an splitting torsion pair in $\mathcal{C}$. Then $\mathcal{Y}= {}^{\perp_{0}}\mathcal{X}$ and $\mathcal{X}=\mathcal{Y}^{\perp_{0}}$.
\end{Remark}
\begin{proof}
Since $\mathrm{Hom}_{\mathcal{C}}(\mathcal{Y},\mathcal{X})=0$ we have that $\mathcal{Y}\subseteq {}^{\perp_{0}} \mathcal{X}$. Now, let $C=X'\oplus Y'\in \mathcal{C}$ with $X'\in \mathcal{X}$ and $Y'\in \mathcal{Y}$  such that 
$$\mathrm{Hom}_{\mathcal{C}}(C,X)=0$$ for all $X\in \mathcal{X}$. Considering the projection $\pi:X'\oplus Y'\longrightarrow X'$ we have that $\pi\in \mathrm{Hom}_{\mathcal{C}}(C,X')=0$. Hence we conclude that $X'=0$ and thus $C=Y'\in\mathcal{Y}$. Proving that $\mathcal{Y}= {}^{\perp_{0}}\mathcal{X}$\\
Now, since $\mathrm{Hom}_{\mathcal{C}}(\mathcal{Y},\mathcal{X})=0$ we have that $\mathcal{X}\subseteq \mathcal{Y}^{\perp_{0}}$. Now, let   $C=X'\oplus Y'\in \mathcal{C}$ with $X'\in \mathcal{X}$ and $Y'\in \mathcal{Y}$  such that 
$$\mathrm{Hom}_{\mathcal{C}}(Y,C)=0$$ for all $Y\in \mathcal{Y}$.  Considering the inclusion $\mu:Y'\longrightarrow X'\oplus Y'$ we have that $\mu\in \mathrm{Hom}_{\mathcal{C}}(Y',C)=0$. Hence we conclude that $Y'=0$, thus $C=X'\in\mathcal{X}$. Proving that $\mathcal{X}=\mathcal{Y}^{\perp_{0}}$.
\end{proof}

We have our first result that characterizes triangular matrix categories via recollements.

\begin{theorem}\label{primerequivalencia}
Suppose we have a recollement
$$\xymatrix{\mathrm{Mod}(\mathcal{C}/I_{\mathcal{B}})\ar[rr]|{\mathrm{res}_{\mathcal{C}}}  &  & \mathrm{Mod}(\mathcal{C})\ar[rr]|{\mathrm{res}_{\mathcal{B}}}\ar@<-2ex>[ll]_{\mathcal{C}/I_{\mathcal{B}}\otimes_{\mathcal{C}}-}\ar@<2ex>[ll]^{\mathcal{C}(\mathcal{C}/I_{\mathcal{B}},-)} &   & \mathrm{Mod}(\mathcal{B})\ar@<-2ex>[ll]_{\mathcal{C}\otimes_{\mathcal{B}}}\ar@<2ex>[ll]^{\mathcal{B}(\mathcal{C},-)}}$$
such that $\mathcal{C}(\mathcal{C}/I_{\mathcal{B}},-) $ is an exact functor. Then, we have an equivalence of categories
$$\mathcal{C}\simeq \left[ \begin{smallmatrix}
 \mathcal{B} & 0 \\ 
\overline{\mathrm{Hom}} &  {}^{\perp_{0}}\mathcal{B}
\end{smallmatrix}\right]$$
\end{theorem}
\begin{proof}
By Remark \ref{2funtores}, the functor $\mathcal{C}(\frac{\mathcal{C}}{I_{\mathcal{B}}},-):\mathrm{Mod}\left(\mathcal{C}\right) \longrightarrow \mathrm{Mod}\left(\mathcal{C}/I_{\mathcal{B}}\right) $ is given as follows: 
$\mathcal{C}(\frac{\mathcal{C}}{I_{\mathcal{B}}},M)(C)=\mathcal{C}\left(\frac{\mathcal{C}(C,-)}{I_{\mathcal{B}}(C,-)},M\right)$ for all $ M\in \mathrm{Mod}(\mathcal{C})$ and  $C\in \mathcal{C}/I_{\mathcal{B}}$ and 
$\mathcal{C}(\frac{\mathcal{C}}{I_{\mathcal{B}}},M)(\overline{f})=\mathcal{C}\left( \frac{\mathcal{C}}{I_{\mathcal{B}}}(f,-),M\right) 
$ for all $\overline{f}=f+I_{\mathcal{B}}(C,C')\in \mathcal{C}(C,C')/I_{\mathcal{B}}(C,C') $ with $C,C'$ in $\mathcal{C}$.\\
If $\mathcal{C}(\mathcal{C}/I_{\mathcal{B}},-) $ is an exact functor, we have that $\frac{\mathcal{C}(C,-)}{I_{\mathcal{B}}(C,-)}$ is a projective $\mathcal{C}$-module. Then we have an  split exact sequence
$$\xymatrix{0\ar[r] & I_{\mathcal{B}}(C,-)\ar[r]^{\psi} & \mathcal{C}(C,-)\ar[r]^{\rho} &  \frac{\mathcal{C}(C,-)}{I_{\mathcal{B}}(C,-)}\ar[r] & 0}$$
Consider $\eta: \mathcal{C}(C,-)\rightarrow  I_{\mathcal{B}}(C,-)$ such that $\eta \psi=1$.
We have the exact sequence 
$$\xymatrix{\mathcal{C}(C,-)\ar[r]^{\psi\eta} & \mathcal{C}(C,-)\ar[r]^{\rho} &  \frac{\mathcal{C}(C,-)}{I_{\mathcal{B}}(C,-)}\ar[r] & 0}$$
Since $\eta \psi=1$, we get that $\theta=\psi\eta:  \mathcal{C}(C,-)\rightarrow \mathcal{C}(C,-)$ is idempotent. Then, there exists an idempotent morphism $\theta':C\rightarrow C$ such that $\theta=\mathcal{C}(\theta',-)$.\\
By Lemma \ref{splitidem}, we have an split exact sequences
$$\xymatrix{0\ar[r] & C''\ar[r]^{\mu''} & C\ar[r]^{p'} & C'\ar[r] & 0}$$
$$\xymatrix{0\ar[r] & C'\ar[r]^{\mu'} & C\ar[r]^{p''} & C''\ar[r] & 0}$$
with $\mu'p'=\theta'$. Then we have split exact sequences
$$\xymatrix{0\ar[r] & \mathcal{C}(C',-)\ar[r]^{\mathcal{C}(p',-)} & \mathcal{C}(C,-)\ar[r]^{\mathcal{C}(\mu'',-)} & \mathcal{C}(C'',-)\ar[r] & 0}$$

$$\xymatrix{0\ar[r] & \mathcal{C}(C'',-)\ar[r]^{\mathcal{C}(p'',-)} & \mathcal{C}(C,-)\ar[r]^{\mathcal{C}(\mu',-)} & \mathcal{C}(C',-)\ar[r] & 0}$$
Since $\mathcal{C}(\mu',-)$ is an epimorphism and $\mu'p'=\theta'$, we have that $\mathcal{C}(\mu'',-)=\mathrm{Coker}\Big(\mathcal{C}(p',-)\circ \mathcal{C}(\mu',-)\Big)=\mathrm{Coker}(\mathcal{C}(\theta',-))=\mathrm{Coker}(\psi\eta)$.\\
Therefore,  we have isomorphism $\lambda:\mathcal{C}(C'',-)\longrightarrow  \frac{\mathcal{C}(C,-)}{I_{\mathcal{B}}(C,-)}$ such that the following diagram commutes
$$\xymatrix{\mathcal{C}(C,-)\ar[r]^{\psi\eta}\ar@{=}[d] & \mathcal{C}(C,-)\ar[r]^{\rho}\ar@{=}[d] &  \frac{\mathcal{C}(C,-)}{I_{\mathcal{B}}(C,-)}\ar[r]& 0\\
\mathcal{C}(C,-)\ar[r]^{\mathcal{C}(\theta',-)} & \mathcal{C}(C,-)\ar[r]^{\mathcal{C}(\mu'',-)}  & \mathcal{C}(C'',-)\ar[r]\ar[u]^{\lambda}  & 0}$$
Then, there exists isomorphism $\delta:\mathcal{C}(C',-)\longrightarrow I_{\mathcal{B}}(C,-)$ such that the following diagram commutes
$$\xymatrix{0\ar[r] & I_{\mathcal{B}}(C,-)\ar[r]^{\psi} & \mathcal{C}(C,-)\ar[r]^{\rho}\ar@{=}[d] &  \frac{\mathcal{C}(C,-)}{I_{\mathcal{B}}(C,-)}\ar[r] & 0\\
0\ar[r] & \mathcal{C}(C',-)\ar[r]^{\mathcal{C}(p',-)}\ar[u]^{\delta}  & \mathcal{C}(C,-)\ar[r]^{\mathcal{C}(\mu'',-)}  & \mathcal{C}(C'',-)\ar[r]\ar[u]^{\lambda} & 0}$$
For $B\in \mathcal{B}$ we have that $\frac{\mathcal{C}(C,-)}{I_{\mathcal{B}}(C,-)}(B)=\frac{\mathcal{C}(C,B)}{I_{\mathcal{B}}(C,B)}=\frac{\mathcal{C}(C,B)}{\mathcal{C}(C,B)}=0$.
Hence we conclude that $\mathcal{C}(C'',B)=0$ for all $B\in \mathcal{B}$. That is,
$C''\in {}^{\perp_{0}}\mathcal{B}=\{X\in \mathcal{C}\mid \mathcal{C}(X,B)=0\,\,\forall B\in \mathcal{B}\}$.
By Yoneda's Lemma we have isomorphism
$$Y:\mathrm{Hom}_{\mathrm{Mod}(\mathcal{C})}\Bigg(\mathcal{C}(C',-),I_{\mathcal{B}}(C,-)\Bigg)\longrightarrow I_{\mathcal{B}}(C,C')$$
such that $Y(\delta)=\alpha:=\delta_{C'}(1_{C'})\in I_{\mathcal{B}}(C,C')$.
Since $\psi_{C'}$ is the inclusion, by the above diagram we have that
$$\alpha=\psi_{C'}\delta_{C'}(1_{C'})=\mathcal{C}(p',-)_{C'}(1_{C'})=p'.$$
Now, since $I_{\mathcal{B}}$ is an ideal, we have that $p'\mu'=1_{C'}\in I_{\mathcal{B}}(C',C')$.
Then, there exists $B\in \mathcal{B}$ such that $1_{C'}=ba$ with $a:C'\rightarrow B$ and $b:B\rightarrow C'$. Therefore, we have an idempotent
$$\gamma:=ab:B\longrightarrow B.$$
It is easy to see that $Ker(1-\gamma)=a$. By Lemma \ref{splitidem}, we have that $C'$ is a direct sumand of $B\in \mathcal{B}$ and since $\mathcal{B}=\mathrm{add}(\mathcal{B})$ we conclude that $C'\in \mathcal{B}$.\\
Hence, we have proved that for each $C\in \mathcal{C}$ there exists $C''\in {}^{\perp_{0}}\mathcal{B}=\{X\in \mathcal{C}\mid \mathcal{C}(X,B)=0\,\,\forall B\in \mathcal{B}\}$ and  $C'\in \mathcal{B}$ such that
$$C\simeq C'\oplus C''.$$
Now, we have that ${}^{\perp_{0}}\mathcal{B}\cap \mathcal{B}=\{0\}$, then we conclude that an indecomposable object in $\mathcal{C}$ belongs to ${}^{\perp_{0}}\mathcal{B}$ or to $\mathcal{B}$. Since $\mathcal{C}$ is Krull-Schmidt, we have an splitting torsion pair $( {}^{\perp_{0}}\mathcal{B}, \mathcal{B})$ in the sense of Definition \ref{defisplitting}. Hence, by
\cite[Proposition 7.1]{LGOS1}, we have an equivalence of categories
$$\mathcal{C}\simeq \left[ \begin{smallmatrix}
 \mathcal{B} & 0 \\ 
\overline{\mathrm{Hom}} &  {}^{\perp_{0}}\mathcal{B}
\end{smallmatrix}\right].$$
\end{proof}
We have our second result that characterizes triangular matrix categories via recollements.

\begin{theorem}\label{segundorsult}
Suppose we have a recollement
$$\xymatrix{\mathrm{Mod}(\mathcal{C}/I_{\mathcal{B}})\ar[rr]|{\mathrm{res}_{\mathcal{C}}}  &  & \mathrm{Mod}(\mathcal{C})\ar[rr]|{\mathrm{res}_{\mathcal{B}}}\ar@<-2ex>[ll]_{\mathcal{C}/I_{\mathcal{B}}\otimes_{\mathcal{C}}-}\ar@<2ex>[ll]^{\mathcal{C}(\mathcal{C}/I_{\mathcal{B}},-)} &   & \mathrm{Mod}(\mathcal{B})\ar@<-2ex>[ll]_{\mathcal{C}\otimes_{\mathcal{B}}}\ar@<2ex>[ll]^{\mathcal{B}(\mathcal{C},-)}}$$
such that $\mathcal{C}/I_{\mathcal{B}}\otimes_{\mathcal{C}}-$ is exact and $\mathcal{B}$ is contravariantly finite subcategory of $\mathcal{C}$. Then there exists equivalence
$$\Xi:\mathcal{C}\longrightarrow \left[ \begin{smallmatrix}
 \mathcal{B}^{\perp_{0}} & 0 \\ 
\overline{\mathrm{Hom}} &  \mathcal{B}
\end{smallmatrix}\right].$$
\end{theorem}
\begin{proof}
We assert that $(\mathcal{B},\mathcal{B}^{\perp_{0}})$ is a splitting torsion pair. Since $\mathcal{B}$ is contravariantly finite, there exists an epimorphism $\mathcal{C}(-,X)\longrightarrow I_{\mathcal{B}}(-,C)\longrightarrow 0$ (see \cite[Proposition 4.12]{RSS}). Hence we can complete the exact sequence
$$\xymatrix{\mathcal{C}(-,X)\ar[r] & \mathcal{C}(-,C)\ar[r] & \mathcal{C}(-,C)/I_{\mathcal{B}}(-,C)\ar[r] & 0}$$
That is, we get that $\frac{\mathcal{C}}{I_{\mathcal{B}}}(-,C)$ is finitely presented and since $\mathcal{C}/I_{\mathcal{B}}\otimes_{\mathcal{C}}-$ is exact we conclude that $\frac{\mathcal{C}}{I_{\mathcal{B}}}(-,C)$ is a flat module for all $C\in \mathcal{C}$. Then, by Proposition \ref{plano=proy}, we get that $\mathcal{C}(-,C)/I_{\mathcal{B}}(-,C)$ is a finitely generated projective module.  Hence, we have that the following exact sequence splits in $\mathrm{Mod}(\mathcal{C}^{op})$:
$$\xymatrix{0\ar[r] & I_{\mathcal{B}}(-,C)\ar[r]  & \mathcal{C}(-,C)\ar[r] & \frac{\mathcal{C}(-,C)}{I_{\mathcal{B}}(-,C)}\ar[r] & 0}$$ 
Similarly to the proof of Proposition \ref{primerequivalencia}, we conclude that there exists $C', C''$ such that $C=C'\oplus C''$  where $\gamma:C'\longrightarrow C$ and  $p:C\longrightarrow C''$  are the inclusion and projection of $C'$ and $C''$ and such that there exists isomorphisms $\lambda$ and $\delta$ such that the following diagram commutes
$$\xymatrix{0\ar[r] & I_{\mathcal{B}}(-,C)\ar[r]^{\psi} & \mathcal{C}(-,C)\ar[r]^{\rho} & \frac{\mathcal{C}(-,C)}{I_{\mathcal{B}}(-,C)}\ar[r] & 0\\
0\ar[r] & \mathcal{C}(-,C')\ar[r]^{\mathcal{C}(-,\gamma)}\ar[u]_{\lambda}  & \mathcal{C}(-,C)\ar[r]^{\mathcal{C}(-,p)}\ar@{=}[u] & \mathcal{C}(-,C'')\ar[r]\ar[u]^{\delta} & 0}$$ 
We assert that  $C'\in \mathcal{B}$ and $C''\in \mathcal{B}^{\perp_{0}}$.  Indeed, for $Z\in \mathcal{B}$, we have that $I_{\mathcal{B}}(Z,C)=\mathcal{C}(Z,C)$. Hence, $0= \frac{\mathcal{C}(Z,C)}{I_{\mathcal{B}}(Z,C)}\simeq \mathcal{C}(Z,C'')$ and thus $C''\in \mathcal{B}^{\perp_{0}}$.\\
On the other hand, by Yoneda's isomorphism
$$Y:\mathrm{Hom}_{\mathrm{Mod}(\mathcal{C})}\Big(\mathcal{C}(-,C'),I_{\mathcal{B}}(-,C)\Big)\longrightarrow I_{\mathcal{B}}(C',C)$$
we have that $Y(\lambda)=\lambda_{C'}(1_{C'})\in I_{\mathcal{B}}(C',C).$ Now, since $\psi_{C'}$ is the inclusion, by the above diagram we have that
$$\lambda_{C'}(1_{C'})=\psi_{C'}(\lambda_{C'}(1_{C'})=\mathcal{C}(-,\gamma)(1_{C'})=\gamma$$
Consequently, $\gamma\in I_{\mathcal{B}}(C',C)$. Since $\gamma:C'\longrightarrow C=C'\oplus C''$ is the inclusion of $C'$, there exists $q:C\longrightarrow C'$ such that $q\gamma=1_{C'}$. Hence, we conclude that $1_{C'}\in I_{\mathcal{B}}(C',C')$. Thus, there exists $\gamma_{1}$ and $\gamma_{2}$ such that the following diagram commutes
$$\xymatrix{C'\ar[rr]^{1}\ar[dr]_{\gamma_{1}} & & C'\\
& B\ar[ur]_{\gamma_{2}} & }$$
with $B\in B$. We get that $\theta:=\gamma_{1}\gamma_{2}:B\longrightarrow B$ is idempotent.
It is easy to see that $\mathrm{Ker}(1-\theta)=\gamma_{1}$. By Lemma \ref{splitidem}, we have that $C'$ is a direct summand of $B\in \mathcal{B}$ and since $\mathcal{B}=\mathrm{add}(\mathcal{B})$ we conclude that $C'\in \mathcal{B}$.\\
Therefore, for all $C\in \mathcal{C}$ we have that  $C=C'\oplus C''$ with $C'\in \mathcal{B}$ and $C''\in \mathcal{B}^{\perp_{0}}$.\\
Since $\mathcal{C}$ is Krull-Schmidt, we have an splitting torsion pair $( \mathcal{B}, \mathcal{B}^{\perp_{0}})$ in the sense of Definition \ref{defisplitting}. Then, by \cite[Proposition 7.1]{LGOS1}, we have an equivalence of categories
$$\mathcal{C}\simeq \left[ \begin{smallmatrix}
\mathcal{B}^{\perp_{0}}& 0 \\ 
\overline{\mathrm{Hom}} &  \mathcal{B}
\end{smallmatrix}\right].$$
\end{proof}

\begin{lemma}\label{exactIm=ker}
Consider the recollement
$$\xymatrix{\mathrm{Mod}(\mathcal{C}/I_{\mathcal{B}})\ar[rr]|{\mathrm{res}_{\mathcal{C}}}  &  & \mathrm{Mod}(\mathcal{C})\ar[rr]|{\mathrm{res}_{\mathcal{B}}}\ar@<-2ex>[ll]_{\mathcal{C}/I_{\mathcal{B}}\otimes_{\mathcal{C}}-}\ar@<2ex>[ll]^{\mathcal{C}(\mathcal{C}/I_{\mathcal{B}},-)} &   & \mathrm{Mod}(\mathcal{B})\ar@<-2ex>[ll]_{\mathcal{C}\otimes_{\mathcal{B}}}\ar@<2ex>[ll]^{\mathcal{B}(\mathcal{C},-)}}$$
If $\mathcal{C}/I_{\mathcal{B}}\otimes_{\mathcal{C}}-$ is exact, then  $\mathrm{Ker}(\mathcal{C}/I_{\mathcal{B}}\otimes_{\mathcal{C}}-)=\mathrm{Im}((\mathcal{C}\otimes_{\mathcal{B}}-))$.
\end{lemma}
\begin{proof}
We know that always happens that $\mathrm{Im}((\mathcal{C}\otimes_{\mathcal{B}}-))\subseteq \mathrm{Ker}(\mathcal{C}/I_{\mathcal{B}}\otimes_{\mathcal{C}}-)$. Let $A\in  \mathrm{Ker}(\mathcal{C}/I_{\mathcal{B}}\otimes_{\mathcal{C}}-)$. By \cite[Proposition 4.9]{Franjou}, we have the following exact sequence
$$\xymatrix{i_{\ast}(L_{1}i^{\ast}(A))\ar[r] & j_{!}j^{\ast}A\ar[r] & A\ar[r] & 0}$$
where $i_{\ast}=res_{\mathcal{C}}$, $L_{1}i^{\ast}$ is the first derived functor of $i^{\ast}=\mathcal{C}/I_{\mathcal{B}}\otimes_{\mathcal{C}}-$, $j_{!}=\mathcal{C}\otimes_{\mathcal{B}}-$ and $j^{\ast}=res_{\mathcal{B}}$.
Since $i^{\ast}$ is exact, we have that $L_{1}i^{\ast}(A)=0$ and thus $j_{!}j^{\ast}A\simeq  A$. Hence, $A\in \mathrm{Im}((\mathcal{C}\otimes_{\mathcal{B}}-))$. Proving that $\mathrm{Ker}(\mathcal{C}/I_{\mathcal{B}}\otimes_{\mathcal{C}}-)=\mathrm{Im}((\mathcal{C}\otimes_{\mathcal{B}}-))$.
\end{proof}

\begin{proposition}\label{ciertaspropiedades}
Consider the recollement
$$\xymatrix{\mathrm{Mod}(\mathcal{C}/I_{\mathcal{B}})\ar[rr]|{\mathrm{res}_{\mathcal{C}}}  &  & \mathrm{Mod}(\mathcal{C})\ar[rr]|{\mathrm{res}_{\mathcal{B}}}\ar@<-2ex>[ll]_{\mathcal{C}/I_{\mathcal{B}}\otimes_{\mathcal{C}}-}\ar@<2ex>[ll]^{\mathcal{C}(\mathcal{C}/I_{\mathcal{B}},-)} &   & \mathrm{Mod}(\mathcal{B})\ar@<-2ex>[ll]_{\mathcal{C}\otimes_{\mathcal{B}}}\ar@<2ex>[ll]^{\mathcal{B}(\mathcal{C},-)}}$$
 Let $(\mathcal{X},\mathcal{Y},\mathcal{Z})$ be the TTF associated to the recollement, where $\mathcal{X}=\mathrm{Ker}(\mathcal{C}/I_{\mathcal{B}}\otimes_{\mathcal{C}}-)$, $\mathcal{Y}=\mathrm{Im}(res_{\mathcal{C}})\simeq \mathrm{Mod}(\mathcal{C}/I_{\mathcal{B}})$  and $\mathcal{Z}=\mathrm{Ker}(\mathcal{C}(\mathcal{C}/I_{\mathcal{B}},-))$.
Then, the following are equivalent.
\begin{enumerate}
\item [(a)] $\mathcal{C}/I_{\mathcal{B}}\otimes_{\mathcal{C}}-$ is exact.

\item [(b)]  $\mathcal{C}(\mathcal{C}/I_{\mathcal{B}},- )\circ (\mathcal{C}\otimes_{\mathcal{B}}-)=0$.

\item [(c)]  $\mathcal{X}\subseteq \mathcal{Z}$.

\item [(d)] $(\mathcal{X},\mathcal{Y})$ is a hereditary torsion pair.
\end{enumerate}
\end{proposition}
\begin{proof}
$(a)\Longleftrightarrow (b)$ follows from Proposition 8.8 in p. 54 in \cite{Franjou}.\\
$(b)\Longrightarrow (c)$.  Suppose that $\mathcal{C}(\mathcal{C}/I_{\mathcal{B}},- )\circ (\mathcal{C}\otimes_{\mathcal{B}}-)=0$. We have already proved that this is equivalent to the fact that $\mathcal{C}/I_{\mathcal{B}}\otimes_{\mathcal{C}}-$ is exact. Hence, by Lemma \ref{exactIm=ker}, we conclude that
$$\mathcal{X}=\mathrm{Ker}(\mathcal{C}/I_{\mathcal{B}}\otimes_{\mathcal{C}}-)=\mathrm{Im}((\mathcal{C}\otimes_{\mathcal{B}}-))\subseteq \mathrm{Ker}(\mathcal{C}(\mathcal{C}/I_{\mathcal{B}},-))=\mathcal{Z}$$
$(c)\Longrightarrow (b)$. We have that $\mathrm{Im}((\mathcal{C}\otimes_{\mathcal{B}}-))\subseteq \mathrm{Ker}(\mathcal{C}/I_{\mathcal{B}}\otimes_{\mathcal{C}}-)=\mathcal{X}\subseteq \mathcal{Z}=\mathrm{Ker}(\mathcal{C}(\mathcal{C}/I_{\mathcal{B}},-))$. Hence $\mathcal{C}(\mathcal{C}/I_{\mathcal{B}},- )\circ (\mathcal{C}\otimes_{\mathcal{B}}-)=0$.\\
$(a) \Longrightarrow (d)$. Suppose that $\mathcal{C}/I_{\mathcal{B}}\otimes_{\mathcal{C}}-$ is exact. 
Let $M\in \mathcal{X}=\mathrm{Ker}(\mathcal{C}/I_{\mathcal{B}}\otimes_{\mathcal{C}}-)$ and $N$ a submodule of $M$.
Since $\mathcal{X}$ is closed under quotients, we have the following exact sequence
$$\xymatrix{0\ar[r] & N\ar[r] & M\ar[r] & M/N\ar[r] & 0}$$
with $M, M/N\in \mathcal{X}$. Hence, by applying the funtor $\mathcal{C}/I_{\mathcal{B}}\otimes_{\mathcal{C}}-$ we obtain the exact sequence
$$\xymatrix{0\ar[r] & \mathcal{C}/I_{\mathcal{B}}\otimes_{\mathcal{C}}N\ar[r] & \mathcal{C}/I_{\mathcal{B}}\otimes_{\mathcal{C}}M=0\ar[r] & \mathcal{C}/I_{\mathcal{B}}\otimes_{\mathcal{C}}M/N=0\ar[r] & 0}$$
We conclude that $N\in  \mathcal{X}=\mathrm{Ker}(\mathcal{C}/I_{\mathcal{B}}\otimes_{\mathcal{C}}-)$, proving that $\mathcal{X}$ is closed under subobjects and hence $(\mathcal{X},\mathcal{Y})$ is a hereditary torsion pair.\\
$(d) \Longrightarrow (a)$. For simplicity we write $i_{\ast}=\mathrm{res}_{\mathcal{C}}$,  $i^{\ast}=\mathcal{C}/I_{\mathcal{B}}\otimes_{\mathcal{C}}-$, $j_{!}=\mathcal{C}\otimes_{\mathcal{B}}-$ and $j^{\ast}=\mathrm{res}_{\mathcal{B}}$.
For $M\in \mathrm{Mod}(\mathcal{C})$, we have the exact sequence (see for example \cite[Proposition 2.8]{Psaro3})
$$\xymatrix{0\ar[r] & \mathrm{Ker}(\mu_{M})\ar[r] & j_{!}j^{\ast}(M)\ar[r]^(.6){\mu_{M}} & M\ar[r]^(.4){\lambda_{M}} & i_{\ast}i^{\ast}(M)\ar[r] & 0}$$
where $\mathrm{Ker}(\mu_{M})\in \mathrm{Im}\Big(\mathrm{res}_{\mathcal{C}}-)\Big)=\mathcal{Y}$. Recall that we always have that $\mathrm{Im}((\mathcal{C}\otimes_{\mathcal{B}}-))\subseteq \mathrm{Ker}(\mathcal{C}/I_{\mathcal{B}}\otimes_{\mathcal{C}}-)$ (see for example \cite[Remark 2.4 (ii)]{Psaro3}).\\
We have that
$j_{!}j^{\ast}(M)\in \mathrm{Im}(j_{!})\subseteq \mathrm{Ker}(i^{\ast})=\mathcal{X}$. Since $(\mathcal{X},\mathcal{Y})$  is hereditary we have that $\mathcal{X}$ is closed under subobjects and hence $\mathrm{Ker}(\mu_{M})\in \mathcal{X}\cap \mathcal{Y}=\{0\}$. Thus, for each $M\in \mathrm{Mod}(\mathcal{C})$ we have the exact sequence
$$\xymatrix{0\ar[r] & j_{!}j^{\ast}(M)\ar[r]^(.6){\mu_{M}} & M\ar[r]^(.4){\lambda_{M}} & i_{\ast}i^{\ast}(M)\ar[r] & 0.}$$
Let 
$$\xymatrix{0\ar[r] & M_{1}\ar[r]^{\alpha} & M_{2}\ar[r]^{\beta}  & M_{3}\ar[r] & 0}$$ 
be an exact sequence in $\mathrm{Mod}(\mathcal{C})$.
Since $i^{\ast}$ and $j_{!}$ are right exact and $i_{\ast}$, $j^{\ast}$ are exact, we have the commutative diagram
$$\xymatrix{ & & 0\ar[d] & & \\
0\ar[r] & j_{!}j^{\ast}(M_{1})\ar[r]^(.6){\mu_{M_{1}}}\ar[d]^{j_{!}j^{\ast}(\alpha)} & M_{1}\ar[r]^(.4){\lambda_{M_{1}}}\ar[d]^{\alpha} & i_{\ast}i^{\ast}(M_{1})\ar[r]\ar[d]^{i_{\ast}i^{\ast}(\alpha)} & 0\\
0\ar[r] & j_{!}j^{\ast}(M_{2})\ar[r]^(.6){\mu_{M_{2}}}\ar[d]^{j_{!}j^{\ast}(\beta)} & M_{2}\ar[r]^(.4){\lambda_{M_{2}}}\ar[d]^{\beta} & i_{\ast}i^{\ast}(M_{2})\ar[r]\ar[d]^{i_{\ast}i^{\ast}(\beta)}  & 0\\
0\ar[r] & j_{!}j^{\ast}(M_{3})\ar[r]^(.6){\mu_{M_{3}}}\ar[d] & M_{3}\ar[r]^(.4){\lambda_{M_{3}}}\ar[d] & i_{\ast}i^{\ast}(M_{3})\ar[r]\ar[d] & 0\\
&0 & 0 & 0}$$

By the snake Lemma applied to the first two columns, we have that the exact sequence
$$\xymatrix{0=\mathrm{Ker}(\mu_{M_{3}})\ar[r] & \mathrm{Coker}(\mu_{M_{1}})\ar[r] & \mathrm{Coker}(\mu_{M_{2}})\ar[r] & \mathrm{Coker}(\mu_{M_{3}})}$$
Hence, $i_{\ast}i^{\ast}(\alpha)$ is a momomorphism an hence  we conclude that we have the exact  sequence
$$\xymatrix{0\ar[r] & i_{\ast}i^{\ast}(M_{1})\ar[r]^{i_{\ast}i^{\ast}(\alpha)} & i_{\ast}i^{\ast}(M_{2})\ar[r]^{i_{\ast}i^{\ast}(\beta)} & i_{\ast}i^{\ast}(M_{3})\ar[r] & 0.}$$
Now, since $i^{\ast}$ is right exact we have the exact sequence
$$\xymatrix{0\ar[r] & K\ar[r] & i^{\ast}(M_{1})\ar[r]^{i^{\ast}(\alpha)} & i^{\ast}(M_{2})\ar[r]^{i^{\ast}(\beta)} & i^{\ast}(M_{3})\ar[r] & 0.}$$
By applying the exact functor $i_{\ast}$, we have the exact sequence 
$$\xymatrix{0\ar[r] & i_{\ast}(K)\ar[r] & i_{\ast}i^{\ast}(M_{1})\ar[r]^{i_{\ast}i^{\ast}(\alpha)} & i_{\ast}i^{\ast}(M_{2})\ar[r]^{i_{\ast}i^{\ast}(\beta)} & i_{\ast}i^{\ast}(M_{3})\ar[r] & 0}$$
Hence, we conclude that $i_{\ast}(K)=0$. Since $i_{\ast}$ is full and faithful, we conclude that $K=0$. Proving that $i^{\ast}=\mathcal{C}/I_{\mathcal{B}}\otimes_{\mathcal{C}}-$ is exact.
\end{proof}

Dually, we have the following

\begin{proposition}
Consider the recollement
$$\xymatrix{\mathrm{Mod}(\mathcal{C}/I_{\mathcal{B}})\ar[rr]|{\mathrm{res}_{\mathcal{C}}}  &  & \mathrm{Mod}(\mathcal{C})\ar[rr]|{\mathrm{res}_{\mathcal{B}}}\ar@<-2ex>[ll]_{\mathcal{C}/I_{\mathcal{B}}\otimes_{\mathcal{C}}-}\ar@<2ex>[ll]^{\mathcal{C}(\mathcal{C}/I_{\mathcal{B}},-)} &   & \mathrm{Mod}(\mathcal{B})\ar@<-2ex>[ll]_{\mathcal{C}\otimes_{\mathcal{B}}}\ar@<2ex>[ll]^{\mathcal{B}(\mathcal{C},-)}}$$
 Let $(\mathcal{X},\mathcal{Y},\mathcal{Z})$ be the TTF associated to the recollement, where $\mathcal{X}=\mathrm{Ker}(\mathcal{C}/I_{\mathcal{B}}\otimes_{\mathcal{C}}-)$, $\mathcal{Y}=\mathrm{Im}(res_{\mathcal{C}})\simeq \mathrm{Mod}(\mathcal{C}/I_{\mathcal{B}})$  and $\mathcal{Z}=\mathrm{Ker}(\mathcal{C}(\mathcal{C}/I_{\mathcal{B}},-))$.
Then, the following are equivalent.
\begin{enumerate}
\item [(a)] $\mathcal{C}(\mathcal{C}/I_{\mathcal{B}},-)$ is exact.

\item [(b)]  $(\mathcal{C}/I_{\mathcal{B}}\otimes_{\mathcal{C}}-) \circ (\mathcal{B}(\mathcal{C},-))=0$.

\item [(c)]  $\mathcal{Z}\subseteq \mathcal{X}$.

\item [(d)] $(\mathcal{Y},\mathcal{Z})$ is a cohereditary torsion pair.
\end{enumerate}
\end{proposition}
\begin{proof}
Dual to Proposition \ref{ciertaspropiedades}.
\end{proof}

\begin{proposition}\label{propsegundorecolle}
Consider the recollement
$$\xymatrix{\mathrm{Mod}(\mathcal{C}/I_{\mathcal{B}})\ar[rr]|{\mathrm{res}_{\mathcal{C}}}  &  & \mathrm{Mod}(\mathcal{C})\ar[rr]|{\mathrm{res}_{\mathcal{B}}}\ar@<-2ex>[ll]_{\mathcal{C}/I_{\mathcal{B}}\otimes_{\mathcal{C}}-}\ar@<2ex>[ll]^{\mathcal{C}(\mathcal{C}/I_{\mathcal{B}},-)} &   & \mathrm{Mod}(\mathcal{B})\ar@<-2ex>[ll]_{\mathcal{C}\otimes_{\mathcal{B}}}\ar@<2ex>[ll]^{\mathcal{B}(\mathcal{C},-)}}$$
such that $\mathcal{C}/I_{\mathcal{B}}\otimes_{\mathcal{C}}-$ is exact.\\ 
Consider the torsion pair $(\mathcal{X},\mathcal{Y})$ where
$\mathcal{X}=\mathrm{Ker}(\mathcal{C}/I_{\mathcal{B}}\otimes_{\mathcal{C}}-)=\mathrm{Im}((\mathcal{C}\otimes_{\mathcal{B}}-))\simeq \mathrm{Mod}(\mathcal{B})$ and $\mathcal{Y}=\mathrm{Im}(\mathrm{res}_{\mathcal{C}})\simeq  \mathrm{Mod}(\mathcal{C}/I_{\mathcal{B}})$.  For simplicity let  $i^{\ast}:=\mathcal{C}/I_{\mathcal{B}}\otimes_{\mathcal{C}}-$, $i_{\ast}:=\mathrm{res}_{\mathcal{C}}$, $j_{!}:=\mathcal{C}\otimes_{\mathcal{B}}-$ and $j^{!}=\mathrm{res}_{\mathcal{B}}$.\\
Let $\mu:j_{!}\circ j^{!}\longrightarrow 1$ the counit of the adjoint pair $(j_{!},j^{!})$ and $\lambda:1\longrightarrow i_{\ast} \circ i^{\ast}$ the unit of the adjoint pair $(i^{\ast},i_{\ast})$.\\
\begin{enumerate}
\item [(a)] For any $M\in \mathrm{Mod}(\mathcal{C})$ canonical sequence associated to the torsion pair $(\mathcal{X},\mathcal{Y})$ is the following
$$\xymatrix{0\ar[r] & j_{!}j^{!}(M)\ar[r]^(.6){\mu_{M}} & M\ar[r]^(.4){\lambda_{M}} & i_{\ast}i^{\ast}(M)\ar[r] & 0}$$

\item [(b)] The  radical associated to the torsion pair $(\mathcal{X},\mathcal{Y})$  is just $\tau:=j_{!}j^{!}$.

\item [(c)] For all $C\in \mathcal{C}$, we have that $I_{\mathcal{B}}(C,-)\simeq \tau(\mathcal{C}(C,-))=j_{!}j^{}\mathcal{C}(C,-)$.

\item [(d)] For all $C\in \mathcal{C}$, we have that $I_{\mathcal{B}}(C,-)=\mathrm{Tr}_{\mathcal{Q}}(\mathcal{C}(C,-))$ where $\mathcal{Q}:=\Big\{\mathcal{C}(B,-)\mid B\in \mathcal{B}\Big\}$.
\end{enumerate}
\end{proposition}
\begin{proof}
$(a)$. We known that there exists exact sequence (see for example \cite[Proposition 2.8]{Psaro3})
$$\xymatrix{0\ar[r] & \mathrm{Ker}(\mu_{M})\ar[r] & j_{!}j^{!}(M)\ar[r]^{\mu_{M}} & M\ar[r]^{\lambda_{M}} & i_{\ast}i^{\ast}(M)\ar[r] & 0}$$
where $\mathrm{Ker}(\mu_{M})\in \mathrm{Im}(i_{\ast})=\mathcal{Y}$. Now, since $i^{\ast}$ is exact and $i^{\ast}j_{!}=0$, we have the following exact sequence
$$\xymatrix{0\ar[r] & i^{\ast}(\mathrm{Ker}(\mu_{M}))\ar[r] & i^{\ast}(j_{!}j^{!}(M))=0}$$
Hence $\mathrm{Ker}(\mu_{M})\in \mathrm{Ker}(i^{\ast})=\mathcal{X}$. Hence we have that $\mathrm{Ker}(\mu_{M})\in \mathcal{X}\cap \mathcal{Y}=0$  hence we have the exact sequence
$$\xymatrix{0\ar[r] & j_{!}j^{!}(M)\ar[r]^{\mu_{M}} & M\ar[r]^{\lambda_{M}} & i_{\ast}i^{\ast}(M)\ar[r] & 0}$$
where $j_{!}j^{!}(M)\in \mathrm{Im}(j_{!})=\mathrm{Ker}(i^{\ast})=\mathcal{X}$ and $i_{\ast}i^{\ast}(M)\in \mathrm{Im}(i_{\ast})=\mathcal{Y}$. We conclude that the above exact sequence  is  the canonical sequence associated to the torsion pair $(\mathcal{X},\mathcal{Y})$.\\
$(b)$. Due to the exact sequence in the item (a), we have that the radical associated to the torsion pair $(\mathcal{X},\mathcal{Y})$  is just $\tau:=j_{!}j^{!}$.\\
$(c)$. Consider the exact sequence 
$$\xymatrix{0\ar[r] & I_{\mathcal{B}}(C,-)\ar[r] & \mathcal{C}(C,-)\ar[r] & \frac{\mathcal{C}(C,-)}{I_{\mathcal{B}}(C,-)}\ar[r] & 0.}$$
We know that $ \mathcal{C}/I_{\mathcal{B}}\otimes_{\mathcal{C}}I_{\mathcal{B}}(C,-)=\frac{I_{\mathcal{B}}(C,-)}{I_{\mathcal{B}}\cdot I_{\mathcal{B}}(C,-)}=0$ since $I_{\mathcal{B}}$ is idempotent.
Hence $I_{\mathcal{B}}(C,-)\in \mathrm{Ker}(i^{\ast})=\mathcal{X}$, clearly we have that $\frac{\mathcal{C}(C,-)}{I_{\mathcal{B}}(C,-)}\in\mathcal{Y}$.
Hence we have that the above exact sequence is isomorphic to the one given by the torsion pair $(\mathcal{X},\mathcal{Y})$ hence we conclude that $I_{\mathcal{B}}(C,-)\simeq \tau(\mathcal{C}(C,-))=j_{!}j^{}\mathcal{C}(C,-))$.\\
$(d)$. By general theory we have that $\tau(M)=\mathrm{Tr}_{\mathcal{X}}(M)$. Since $\mathcal{X}\simeq \mathrm{Mod}(\mathcal{B})$ via $j_{!}=\mathcal{C}\otimes_{\mathcal{B}}-$ and $\mathcal{P}:=\{\mathcal{B}(B,-)\}_{B\in \mathcal{B}}$ is a set of projective generators of $\mathrm{Mod}(\mathcal{B})$, we have that 
$$\mathcal{Q}:=\Big\{\mathcal{C}\otimes_{\mathcal{B}}\mathcal{B}(B,-)=\mathcal{C}(B,-)\mid B\in \mathcal{B}\Big\}$$
is a set of projective generators of $\mathcal{X}$. Hence for $M\in \mathcal{X}$ there exists an epimorphism
$$\Big(\bigoplus_{B\in \mathcal{B}}\mathcal{C}(B,-)\Big)^{(I)}\longrightarrow M$$
Therefore, we conclude that
$$\tau(M)=\mathrm{Tr}_{\mathcal{X}}(M)=\mathrm{Tr}_{\mathcal{Q}}(M).$$
In particular, we have that $I_{\mathcal{B}}(C,-)=\tau(\mathcal{C}(C,-))=\mathrm{Tr}_{\mathcal{Q}}(\mathcal{C}(C,-).$
\end{proof}

\begin{proposition}
Consider the recollement
$$\xymatrix{\mathrm{Mod}(\mathcal{C}/I_{\mathcal{B}})\ar[rr]|{\mathrm{res}_{\mathcal{C}}}  &  & \mathrm{Mod}(\mathcal{C})\ar[rr]|{\mathrm{res}_{\mathcal{B}}}\ar@<-2ex>[ll]_{\mathcal{C}/I_{\mathcal{B}}\otimes_{\mathcal{C}}-}\ar@<2ex>[ll]^{\mathcal{C}(\mathcal{C}/I_{\mathcal{B}},-)} &   & \mathrm{Mod}(\mathcal{B}).\ar@<-2ex>[ll]_{\mathcal{C}\otimes_{\mathcal{B}}}\ar@<2ex>[ll]^{\mathcal{B}(\mathcal{C},-)}}$$
 Let $(\mathcal{X},\mathcal{Y},\mathcal{Z})$ be the TTF associated to the recollement, where $\mathcal{X}=\mathrm{Ker}(\mathcal{C}/I_{\mathcal{B}}\otimes_{\mathcal{C}}-)$, $\mathcal{Y}=\mathrm{Im}(res_{\mathcal{C}})\simeq \mathrm{Mod}(\mathcal{C}/I_{\mathcal{B}})$  and $\mathcal{Z}=\mathrm{Ker}(\mathcal{C}(\mathcal{C}/I_{\mathcal{B}},-))$. If $\mathcal{X}=\mathcal{Z}$ hence 
\begin{enumerate}
\item [(a)] $\mathrm{Mod}(\mathcal{C})\simeq \mathrm{Mod}(\mathcal{B})\times \mathrm{Mod}(\mathcal{C}/I_{\mathcal{B}}).$

\item [(b)] $\mathcal{C}\simeq \left[ \begin{smallmatrix}
\mathcal{B} & 0 \\ 
0 &  {}^{\perp_{0}}\mathcal{B}
\end{smallmatrix}\right].$
\end{enumerate}
 \end{proposition}
 \begin{proof}
$(a)$. We have that $\mathrm{Im}(\mathcal{B}(\mathcal{C},-))\subseteq \mathrm{Ker}(\mathcal{C}(\mathcal{C}/I_{\mathcal{B}},-))=\mathcal{Z}=\mathcal{X}=\mathrm{Ker}(\mathcal{C}/I_{\mathcal{B}}\otimes_{\mathcal{C}}-)$.
Hence, $$(\mathcal{C}/I_{\mathcal{B}}\otimes_{\mathcal{C}}-) \circ (\mathcal{B}(\mathcal{C},-))=0.$$
By \cite[ Proposition 8.8]{Franjou}, we have that $\mathcal{C}(\mathcal{C}/I_{\mathcal{B}},-)$ is an exact functor. Then, by  Proposition 3.9 in p. 87 in \cite{Psaro3}, we have a splitting recollement and hence
$$\mathrm{Mod}(\mathcal{C})\simeq \mathrm{Mod}(\mathcal{B})\times \mathrm{Mod}(\mathcal{C}/I_{\mathcal{B}}).$$
$(b)$.  In item (a) we showed that $\mathcal{C}(\mathcal{C}/I_{\mathcal{B}},-)$ is an exact functor. Then, by  Proposition \ref{primerequivalencia}, we have an equivalence $$\mathcal{C}\simeq \left[ \begin{smallmatrix}
 \mathcal{B} & 0 \\ 
M &  {}^{\perp_{0}}\mathcal{B}
\end{smallmatrix}\right],$$
where $M:=\overline{\mathrm{Hom}}$. Recall that $\mathcal{C}/I_{\mathcal{B}}\simeq {}^{\perp_{0}}\mathcal{B}$.\\
Hence, the given recollement is equivalent to the one given in the Proposition  \ref{primerrecolle}. By  \cite[ Proposition 5.4]{LGOS1}, we get that $\mathrm{Mod}(\mathcal{C})$ is equivalent to the comma category $\Big(\mathbb{F}\mathrm{Mod}(\mathcal{B}), \mathrm{Mod}({}^{\perp_{0}}\mathcal{B})\Big)$. It can be seen that $\mathbb{F}\simeq \Big(\mathcal{C}(\mathcal{C}/I_{\mathcal{B}},-)\Big)\circ \Big(\mathcal{C}\otimes_{\mathcal{B}}-\Big)=0$.
By  the construction of  $\mathbb{F}$ in  \cite{LGOS1} in p. 861, we have that $\mathbb{F}\circ Y=E$ where $E:\mathcal{B}\longrightarrow \mathrm{Mod}({}^{\perp_{0}}\mathcal{B})$ is the functor defined in definition 3.1 in  \cite{LGOS1} and $Y$ is the contravariant Yoneda functor. Hence we conclude that $E=0$ and hence $M=\overline{\mathrm{Hom}}=0$ (see  \cite[Definition 3.1]{LGOS1}). Then  $\mathcal{C}\simeq \left[ \begin{smallmatrix}
 \mathcal{B} & 0 \\ 
0 &  {}^{\perp_{0}}\mathcal{B}
\end{smallmatrix}\right]\simeq   \mathcal{B} \times  ({}^{\perp_{0}}\mathcal{B})\simeq   \mathcal{B} \times (\mathcal{C}/I_{\mathcal{B}})$.

 \end{proof}

The following result  is due to  Zimmermann (see Lemma 2.1.9 and  Corollary 2.1.10 in \cite{Trlifaj}).

\begin{lemma}\label{LemZimm}
Let $\mathcal{A}$ be an additive category with splitting idempotents. Suppose that $\mathcal{C}$ 
is a class of objects from $\mathcal{A}$ closed under finite direct sums and direct summands.
\begin{enumerate}
\item [(a)] Let $f\in \mathrm{Hom}(M,C)$ be a $\mathcal{C}$-preenvelope of $M$. Let $E=\mathrm{End}(C)$ and $I=\{g\in E\mid gf=0\}$. Assume that idempotents lift modulo $\mathrm{Rad}(E)$ and that there exists a left idel $J$ of $E$ such that $I+J=E$ and $I\cap J\subseteq \mathrm{Rad}(E)$. Then $M$ has a $\mathcal{C}$-envelope.

\item [(b)] Let $f\in \mathrm{Hom}(C,M)$ be a $\mathcal{C}$-precover of $M$. Let $E=\mathrm{End}(C)$ and $I=\{g\in E\mid fg=0\}$. Assume that idempotents lift modulo $\mathrm{Rad}(E)$ and that there exists a right ideal $J$ of $E$ such that $I+J=E$ and $I\cap J\subseteq \mathrm{Rad}(E)$. Then $M$ has a $\mathcal{C}$-cover.
\end{enumerate}
\end{lemma}

\begin{corollary}
Let $\mathcal{A}$ be an additive category with splitting idempotents. Suppose that $\mathcal{C}$ is a 
is a class of objects from $\mathcal{A}$ closed under finite direct sums and direct summands.
\begin{enumerate}
\item [(a)] Assume that $M$ has a $\mathcal{C}$-preenvelope $f\in \mathrm{Hom}(M,C)$ such that $R=\mathrm{End}(C)$ is a semiperfect ring. Then $M$ has  a $\mathcal{C}$-envelope.

\item [(b)] Assume that $M$ has a $\mathcal{C}$-precover $f\in \mathrm{Hom}(C,M)$ such that $\mathrm{End}(C)$ is a semiperfect ring. Then $M$ has  a $\mathcal{C}$-cover.
\end{enumerate}
\end{corollary}

\begin{lemma}\label{lematorpair}
Let $\mathcal{D}$ be a Grothendieck abelian category. Let  $(\mathcal{T},\mathcal{F})$ be a torsion pair in $\mathcal{D}$ such that   $\mathcal{T}$ and $\mathcal{F}$ are abelian subcategories of $\mathcal{D}$.  The following holds.
\begin{enumerate}
\item [(a)] $\mathcal{T}$ is closed under subobjects and $\mathcal{F}$ is closed under quotients. That is $(\mathcal{T},\mathcal{F})$ is a hereditary and cohereditary torsion pair.

\item [(b)] There exists subcategory  $\mathcal{Z}$ of $\mathcal{D}$ such that $(\mathcal{T},\mathcal{F},\mathcal{Z})$ is a TTF.

\item [(c)] If $\mathcal{T}$  is closed under products, there exists subcategory $\mathcal{W}$ of $\mathcal{D}$ such that $(\mathcal{W},\mathcal{T},\mathcal{F})$ is a TTF.
\end{enumerate}
\end{lemma}
\begin{proof}
$(a)$. Let $Y\in \mathcal{T}$ and $i:X\longrightarrow Y$ be a monomorphism in $\mathcal{D}$. Consider the epimorphism $\pi:Y\longrightarrow Y/X$ in $\mathcal{D}$. Since $(\mathcal{T},\mathcal{F})$ is a torsion pair, we have that $\mathcal{T}$ is closed under quotients, coproducts and extensions. So we have that $Y/X\in \mathcal{T}$.
Since $\mathcal{T}$ is an abelian full subcategory of $\mathcal{D}$, we have that the inclusion funtor
$I:\mathcal{T}\longrightarrow \mathcal{D}$ is an exact funtor.\\
Since $\mathcal{T}$ is abelian, we have an exact sequence in $\mathcal{T}$:
$$\xymatrix{0\ar[r] & K\ar[r] & Y\ar[r] & Y/X\ar[r] & 0}$$
Since $I$ is an exact functor, we have that
$$\xymatrix{0\ar[r] & K\ar[r] & Y\ar[r] & Y/X\ar[r] & 0}$$
is exact in $\mathcal{D}$. Hence, $X\simeq K\in \mathcal{T}$. That is $\mathcal{T}$ is closed under subobjects.\\
Let us see now that $\mathcal{F}$ is closed under quotients.
Let $Z\in \mathcal{F}$ and $i:N\longrightarrow Z$ a monomorphism. Since $\mathcal{F}$ is a torsion-free class we have that $N\in \mathcal{F}$.  Since $\mathcal{F}$ is an abelian subcategory of $\mathcal{D}$ we have that the inclusion functor is exact. We have an exact sequence in $\mathcal{F}$:
$$\xymatrix{0\ar[r] & N\ar[r] & Z\ar[r] & Q\ar[r] & 0.}$$
Hence the last exact sequence is exact in $\mathcal{D}$ and hence we conclude that
$Z/N\simeq Q\in \mathcal{F}$. So, $\mathcal{F}$ is closed under quotients.\\
$(b)$  Since $(\mathcal{T},\mathcal{F})$ is a torsion pair, we have that $\mathcal{F}$ is closed under subobjects, products and  extensions. Now, let $\{C_{i}\}_{i\in I}$ be a family of objects in $\mathcal{F}$, we have a monomorphism (see \cite[Proposition 1.1]{MitBook} in p. 81)
$$\mu:\bigoplus_{i\in I}C_{i}\longrightarrow \prod_{i\in I}C_{i}.$$
Since $\mathcal{F}$ is closed under subobjects and products, we have that $\bigoplus_{i\in I}C_{i}\in \mathcal{F}$. By item (a) we have that $\mathcal{F}$ is closed under quotients. Hence $\mathcal{F}$ is closed under quotients, coproducts and extensions. Hence, there exists  $\mathcal{Z}$ such that $(\mathcal{F},\mathcal{Z})$ is a torsion pair.\\
$(c)$. By item $(a)$, we have that $\mathcal{T}$ is closed under subobjects. By hypothesis, we have that $\mathcal{T}$ is closed under products. Since  $(\mathcal{T},\mathcal{F})$ is a torsion pair in $\mathrm{Mod}(\mathcal{C})$ we also have that $\mathcal{T}$  is closed under extensions. Hence $\mathcal{T}$ is a torsion-free class for some torsion theory $(\mathcal{W},\mathcal{T})$. Then we have that $(\mathcal{W},\mathcal{T},\mathcal{F})$ is a TTF torsion theory.
\end{proof}

In the following results we will give another characterization of triangular matrix categories, which is a generalization of the  Proposition 3.6 in \cite{LipingLi}.

\begin{definition}\label{definitionXY}
Let $\mathcal{C}$ be a Krull-Schmidt category with splitting idempotents. Suppose there exists a torsion pair
$(\mathcal{T},\mathcal{F})$ in $\mathrm{Mod}(\mathcal{C})$ where $\mathcal{T}$ and $\mathcal{F}$ are nontrivial abelian subcategories of $\mathrm{Mod}(\mathcal{C})$. Let $\tau$ the radical associated to the torsion pair $(\mathcal{T},\mathcal{F})$. For  each $C\in \mathcal{C}$ consider its decomposition into irreducibles $C=\bigoplus_{i=1}^{n}C_{i}.$
We define the following
$$\mathcal{X}:=\Big\{C=\bigoplus_{i=1}^{n}C_{i}\in \mathcal{C}\,\, \Big|\,\,\,  \frac{\mathcal{C}(C_{i},-)}{\tau(\mathcal{C}(C_{i},-))}\neq 0 \,\, \forall i\Big\},$$
$$\mathcal{Y}:=\Big\{C\in \mathcal{C}\,\, \Big| \,\, \,  \mathcal{C}(C,-)\in \mathcal{T}\Big\}=\Big\{C\in \mathcal{C}\,\, \Big| \,\, \,  \frac{\mathcal{C}(C,-)}{\tau(\mathcal{C}(C,-))}=0\Big\}.$$
\end{definition}

\begin{lemma}\label{lemamorcero}
Let $\mathcal{C}$ be a Krull-Schmidt category with splitting idempotents. Suppose there exists a torsion pair
$(\mathcal{T},\mathcal{F})$ in $\mathrm{Mod}(\mathcal{C})$ where $\mathcal{T}$ and $\mathcal{F}$ are nontrivial abelian subcategories of $\mathrm{Mod}(\mathcal{C})$. Let $\tau$ the radical associated to the torsion pair $(\mathcal{T},\mathcal{F})$.\\
Let $C\in \mathcal{C}$ indecomposable such that  $\frac{\mathcal{C}(C,-)}{\tau(\mathcal{C}(C,-))}\neq 0$. Hence,
$$\mathrm{Hom}_{\mathrm{Mod}(\mathcal{C})}(\mathcal{C}(C,-)),M)=0,\quad \quad \forall M\in \mathcal{T}.$$
\end{lemma}
\begin{proof}
Let $C\in \mathcal{C}$ indecomposable. Consider the canonical exact sequence
$$\xymatrix{0\ar[r] & \tau(\mathcal{C}(C,-))\ar[r]^{i}& \mathcal{C}(C,-)\ar[r]^{\pi} & \frac{\mathcal{C}(C,-)}{\tau(\mathcal{C}(C,-))}\ar[r] & 0}$$
where $\tau(\mathcal{C}(C,-))\in \mathcal{T}$ and  $0\neq \frac{\mathcal{C}(C,-)}{\tau(\mathcal{C}(C,-))}\in \mathcal{F}$.
Since $\mathrm{End}( \mathcal{C}(C,-), \mathcal{C}(C,-))=\mathrm{End}_{\mathcal{C}}(C,C)$ is semiperfect (see \cite[Corollary 4.4]{Krause1} in p. 546) and $ \mathcal{C}(C,-)$ is projective, by a result due to  Zimmermann (see Lemma \ref{LemZimm}), we have that $ \frac{\mathcal{C}(C,-)}{\tau(\mathcal{C}(C,-))}$ has a projective cover in $\mathrm{Mod}(\mathcal{C})$. Since $\mathcal{C}(C,-)$ is indecomposable, we conclude that if $\frac{\mathcal{C}(C,-)}{\tau(\mathcal{C}(C,-))}\neq 0$, then $\pi$ is the projective cover of $\frac{\mathcal{C}(C,-)}{\tau(\mathcal{C}(C,-))}$  (see also \cite[Proposition 2.1(c)]{AusM2}).\\

Let $f:\mathcal{C}(C,-)\longrightarrow M$ be with $M\in \mathcal{T}$. We consider the morphism $fi:\tau(\mathcal{C}(C,-))\longrightarrow M$ and its factorization through its image 
$$\xymatrix{\tau(\mathcal{C}(C,-))\ar[r]^(.6){p} & L\ar[r]^{u} & M.}$$ Hence, we have the following commutative diagram

$$\xymatrix{0\ar[r] & \tau(\mathcal{C}(C,-))\ar[r]^{i}\ar[d]^{p} & \mathcal{C}(C,-)\ar[r]^{\pi}\ar[d]^{f} & \frac{\mathcal{C}(C,-)}{\tau(\mathcal{C}(C,-))}\ar[r] & 0\\
0\ar[r] & L\ar[r]^{u} & M\ar[r] & \frac{M}{L}\ar[r] & 0}$$

So, there exists $h: \frac{\mathcal{C}(C,-)}{\tau(\mathcal{C}(C,-))}\longrightarrow \frac{M}{L}$ such that the above diagrama commutes. As in the proof of Prop. 3.6 in  \cite{LipingLi}, we have that $\mathrm{Hom}(\mathcal{F},\mathcal{T})=0$. Since $\frac{\mathcal{C}(C,-)}{\tau(\mathcal{C}(C,-))}\in \mathcal{F}$ and $\frac{M}{L}\in \mathcal{T}$ (since $\mathcal{T}$ is closed under quotients, see Lemma \ref{lematorpair}),  we conclude that $h=0$.\\

Since  $p$ is an epimorphism, by the snake Lemma we have the following commutative diagram
$$\xymatrix{ & 0\ar[d] & 0\ar[d] & 0\ar[d]\\
0\ar[r] & \mathrm{Ker}(p)\ar[r]\ar[d]^{\alpha} & \mathrm{Ker}(f)\ar[r]^(.4){\theta}\ar[d]^{\beta} & \mathrm{Ker}(h)=\frac{\mathcal{C}(C,-)}{\tau(\mathcal{C}(C,-))}\ar[r]\ar@{=}[d] & 0\\
0\ar[r] & \tau(\mathcal{C}(C,-))\ar[r]^{i} & \mathcal{C}(C,-)\ar[r]^{\pi}& \frac{\mathcal{C}(C,-)}{\tau(\mathcal{C}(C,-))}\ar[r] & 0.}$$
Then, we have  $\theta$ is an epimorphism and $ \theta=\pi\beta$. We conclude that  $\beta$ is an epimorphism  and thus an isomorphism (since $\pi$ is essential because $\pi$ is the projective cover of  $\frac{\mathcal{C}(C,-)}{\tau(\mathcal{C}(C,-))}\neq 0$.) Therefore,
$f=0$  since $f\beta=0$.
\end{proof}

\begin{lemma}\label{novacios}
Consider the classes $\mathcal{X}$ and $\mathcal{Y}$ given un Definition \ref{definitionXY}. Then $\mathcal{Y}\neq \emptyset$, $\mathcal{X}\neq \emptyset$ and  $\mathrm{Hom}(\mathcal{Y},\mathcal{X})=0$.
\end{lemma}
\begin{proof}
Suppose that $C\in \mathcal{Y}$ for all $C\in \mathcal{C}$ indecomposable. Then, $\mathcal{C}(C,-)\in \mathcal{T}$ for all $C\in \mathcal{C}$ (not necessarely indecomposable). Let $M\in \mathcal{F}$ be and $C\in \mathcal{C}$. Hence,  since $(\mathcal{T},\mathcal{F})$ is a torsion pair we have that
$$M(C)\simeq \mathrm{Hom}_{\mathrm{Mod}(\mathcal{C})}(\mathcal{C}(C,-),M)=0.$$
Hence, $M=0$ and thus $\mathcal{F}=0$, which is a contradiction. Hence, there  exists indecomposable $C\in \mathcal{C}$ such that $C\notin \mathcal{Y}$ and thus $C\in \mathcal{X}$. Proving that $\mathcal{X}\neq \emptyset$.\\
Now, suppose that $C\in \mathcal{X}$ for all $C\in \mathcal{C}$ indecomposable. Let $M\in \mathcal{T}$ and $C\in \mathcal{C}$ an indecomposable. Since $C\in \mathcal{X}$, we have that $\frac{\mathcal{C}(C,-)}{\tau(\mathcal{C}(C,-))}\neq 0$,  hence by lemma \ref{lemamorcero}, we have that
 $$M(C)\simeq \mathrm{Hom}(\mathcal{C}(C,-), M)=0.$$
 Next, for arbitrary $C$ consider its decomposition into indecomposables $C=\bigoplus_{i=1}^{n}C_{i}$. Hence, we have that $M(C)=\bigoplus_{i=1}^{n}M(C_{i})=0$. Thus, $M=0$ and thus $\mathcal{T}=0$, which is a contradition. Hence, there exists indecomposable $C$ that not belongs to $\mathcal{X}$,  and thus $C\in \mathcal{Y}$. Proving that $\mathcal{Y}\neq \emptyset$.\\
We proceed to show that $\mathrm{Hom}(\mathcal{Y},\mathcal{X})=0$.
Indeed, let $C'\in \mathcal{Y}$ and $C\in \mathcal{X}$.  Consider the decomposition of $C$ into indecomposables $C=\bigoplus_{i=1}^{n}C_{i}$. Since $C\in \mathcal{X}$, we have that $\frac{\mathcal{C}(C_{i},-)}{\tau(\mathcal{C}(C_{i},-))}\neq 0$ for all $i$. Hence, by Yoneda's Lemma, Lemma \ref{lemamorcero} and the fact that $\mathcal{C}(C',-)\in \mathcal{T}$ we have that
\begin{align*}
\mathrm{Hom}(C',C)=\mathrm{Hom}\Big(\mathcal{C}(C,-),\mathcal{C}(C',-)\Big) & =\mathrm{Hom}\Big(\bigoplus_{i=1}^{n}\mathcal{C}(C_{i},-),\mathcal{C}(C',-)\Big)\\
& =\bigoplus_{i=1}^{n}\mathrm{Hom}\Big(\mathcal{C}(C_{i},-),\mathcal{C}(C',-)\Big)\\
& =0.
\end{align*}
\end{proof}

We now give a characterization of triangular matrix categories in terms of torsion pairs.

\begin{theorem}\label{mismoteorema}
Let $\mathcal{C}$ be a Krull-Schmidt category with splitting idempotents. Suppose there exists a torsion pair
$(\mathcal{T},\mathcal{F})$ in $\mathrm{Mod}(\mathcal{C})$ where $\mathcal{T}$ and $\mathcal{F}$ are nontrivial abelian subcategories of $\mathrm{Mod}(\mathcal{C})$. Hence, we have an equivalence
$$\mathcal{C}\simeq \left[ \begin{smallmatrix}
 \mathcal{B} & 0 \\ 
\overline{\mathrm{Hom}} &  {}^{\perp_{0}}\mathcal{B}
\end{smallmatrix}\right]$$
for some subcategory $\mathcal{B}$ of $\mathcal{C}$.
\end{theorem}
\begin{proof}
Consider the following subcategories of $\mathcal{C}$:
$$\mathcal{X}:=\Big\{C=\bigoplus_{i=1}^{n}C_{i}\in \mathcal{C}\,\, \Big|\,\,\,  \frac{\mathcal{C}(C_{i},-)}{\tau(\mathcal{C}(C_{i},-))}\neq 0 \,\, \forall i\Big\},$$
$$\mathcal{Y}:=\Big\{C\in \mathcal{C}\,\, \Big| \,\, \,  \mathcal{C}(C,-)\in \mathcal{T}\Big\}=\Big\{C\in \mathcal{C}\,\, \Big| \,\, \,  \frac{\mathcal{C}(C,-)}{\tau(\mathcal{C}(C,-))}=0\Big\}.$$
By Lemma \ref{novacios}, we have that $\mathcal{Y}\neq \emptyset$, $\mathcal{X}\neq \emptyset$ and  $\mathrm{Hom}(\mathcal{Y},\mathcal{X})=0$. It is clear that if $C$ is an indecomposable object in $\mathcal{C}$, we have that $C\in \mathcal{X}$ or $C\in \mathcal{Y}$. Hence, we have that $(\mathcal{Y},\mathcal{X})$ is a splitting torsion pair in $\mathcal{C}$. Hence, by
\cite[Proposition 7.1]{LGOS1}, we have an equivalence of categories
$$\mathcal{C}\simeq \left[ \begin{smallmatrix}
 \mathcal{X} & 0 \\ 
\overline{\mathrm{Hom}} &  \mathcal{Y}
\end{smallmatrix}\right].$$
\end{proof}

\section{Appendix: On flat modules an finitely presented modules}
\subsection{Introduction}
In this appendix we recollect some well known results of the finitely presented functors which are folklore and well known for experts but we could not find a suitable reference for the proofs of such a results. First we recall, that we have  the following classical result is due to H. Lenzing in \cite{Lenzing1}.

\begin{proposition}
\begin{enumerate}
\item [(a)] A right R-module $A$ is finitely generated if
and only if for any collection $\{C_{\alpha}\}$ of left $R$-modules the natural map
$$A\otimes \prod C_{\alpha}\longrightarrow \prod (A\otimes C_{\alpha})$$
is surjective.
\item [(b)] A right R-module $A$ is finitely presented if
and only if for any collection $\{C_{\alpha}\}$ of left $R$-modules the natural map
$$A\otimes \prod C_{\alpha}\longrightarrow \prod (A\otimes C_{\alpha})$$
is an isomorphism.
\end{enumerate}
\end{proposition}

By reviewing the paper \cite{Lenzing1}, we have the following remark, which is interesting but it is written in German, and that is why we are reproducing here.

\begin{Remark}
In Satz 1 in  the article \cite{Lenzing1}, H. Lenzing  proved that $A$ is finitely generated if and only if
$$ A\otimes R^{I}\longrightarrow (A\otimes R)^{I}=A^{I}$$ is surjective for all set $I$.\\
Moreover, in Remark 1 in  \cite{Lenzing1},  it is stated the following: For a module $A$ the following conditions are equivalent
\begin{enumerate}
\item [(a)]  $ A\otimes R^{I}\longrightarrow (A\otimes R)^{I}=A^{I}$ is surjective for all $I$.
\item [(b)] $A^{I}=\bigcup_{N} N^{I}$  where $N$ run over all the finitely generated submodules of $A$.
\item [(c)] Each family  $(x_{\alpha})_{\alpha\in I}\in A^{I}$ is contained in a finitely generated submodule of $A$.
\end{enumerate}
\end{Remark}

In this appendix we will give a proof of Lenzing's result in the context of $\mathrm{Mod}(\mathcal{C})$ for an additive category $\mathcal{C}$ (this results must be known for the experts).\\

\subsection{Flat modules and finitely presented modules}

The following is the generalization of Satz 1 in \cite{Lenzing1}.

\begin{proposition}\label{criteriofg}
Let $F\in \mathrm{Mod}(\mathcal{C}^{op})$ be and consider the functor  $F\otimes_{\mathcal{C}}-:\mathrm{Mod}(\mathcal{C})\longrightarrow \mathrm{Ab}$. The following conditions are equivalent.
\begin{enumerate}
\item [(a)] For every family of objets $\{M_{i}\}_{i\in I}$ in $\mathrm{Mod}(\mathcal{C})$ the canonical morphism $\delta:F\otimes_{\mathcal{C}}(\prod_{i\in I}M_{i})\longrightarrow \prod_{i\in I} F\otimes_{\mathcal{C}}M_{i}$ is surjective. 

\item [(b)]  $F$ is finitely generated.
\end{enumerate}
\end{proposition}
\begin{proof}
$(a)\Longrightarrow (b)$. We have an epimorphism  in $\mathrm{Mod}(\mathcal{C}^{op})$:
$$\xymatrix{ \bigoplus_{i\in I}\mathcal{C}(-,C_{i})\ar[r]^(.7){\varphi} & F\ar[r] & 0.}$$
Consider $\prod_{i\in I}\mathcal{C}(C_{i},-)\in \mathrm{Mod}(\mathcal{C})$ and $\pi_{i}:\prod_{i\in I}\mathcal{C}(C_{i},-)\longrightarrow  \mathcal{C}(C_{i},-)$ the canonical projections. Then we have a family of morphisms in $\text{Ab}$:
$$\left\{F\otimes \pi_{i}:F\otimes \Big(\prod_{i\in I}\mathcal{C}(C_{i},-)\Big)\longrightarrow F\otimes \mathcal{C}(C_{i},-)\right\}_{i\in I}$$
Then, there exists a unique morphism  $\delta:F\otimes (\prod_{i\in I}\mathcal{C}(C_{i},-))\longrightarrow \prod_{i\in I}(F\otimes \mathcal{C}(C_{i},-))$  such that the following diagram commutes
$$\xymatrix{F\otimes \Big(\prod_{i\in I}\mathcal{C}(C_{i},-)\Big)\ar[rr]^{\delta}\ar[dr]_{F\otimes \pi_{i}} & & \prod_{i\in I}(F\otimes \mathcal{C}(C_{i},-))\ar[dl]_{p_{i}}\\
 & F\otimes \mathcal{C}(C_{i},-) &}$$
 where $p_{i}: \prod_{i\in I}(F\otimes \mathcal{C}(C_{i},-))\longrightarrow F\otimes \mathcal{C}(C_{i},-)$ are the canonical projections of the product in $\text{Ab}$. By hipothesis, we have that $\delta$ is surjective.\\
On the other hand, we have an epimorphism
$$\xymatrix{\bigoplus_{\lambda \in \Lambda}\mathcal{C}(C_{\lambda},-)\ar[r]^{\Theta} & \prod_{i\in I}\mathcal{C}(C_{i},-)\ar[r] & 0}$$
Since $F\otimes_{\mathcal{C}}-$ is right exact we have the epimorphism
$$\xymatrix{ F\otimes \Big(\bigoplus_{\lambda \in \Lambda}\mathcal{C}(C_{\lambda},-)\Big)\ar[r]^{F\otimes \Theta} & F\otimes (\prod_{i\in I}\mathcal{C}(C_{i},-))\ar[r] & 0.}$$
Since $F\otimes_{\mathcal{C}}-$ preserve arbitrary colimits and we have that $F\otimes \Big(\bigoplus_{\lambda \in \Lambda}\mathcal{C}(C_{\lambda},-)\Big)\simeq \bigoplus_{\lambda \in \Lambda}F\otimes\mathcal{C}(C_{\lambda},-)$.
Hence, we have an epimorphism
$$\xymatrix{ \bigoplus_{\lambda \in \Lambda}F\otimes_{\mathcal{C}}\mathcal{C}(C_{\lambda},-)\ar[rr]^{\delta\circ (F\otimes \Theta)}&&  \prod_{i\in I}F\otimes_{\mathcal{C}}\mathcal{C}(C_{i},-)\ar[r] & 0}$$
Since $F\otimes_{\mathcal{C}}\mathcal{C}(C_{\lambda},-)=F(C_{\lambda})$ and $F\otimes_{\mathcal{C}}\mathcal{C}(C_{i},-)=F(C_{i})$, we identify the above epimorphism with
$$\xymatrix{ \bigoplus_{\lambda \in \Lambda}F(C_{\lambda})\ar[r]^{\Gamma} & \prod_{i\in I}F(C_{i})\ar[r] & 0}$$
where $\Gamma=\delta\circ (F\otimes \Theta)$.\\
Let $u_{i}:\mathcal{C}(-,C_{i})\longrightarrow \bigoplus_{i\in I}\mathcal{C}(-,C_{i})$ be the canonical inclusions of the coproduct in $\mathrm{Mod}(\mathcal{C}^{op})$. Now, via the Yoneda's isomorphism we have isomorphisms
$$Y_{i}:\mathrm{Hom}_{\mathrm{Mod}(\mathcal{C}^{op})}\Big(\mathcal{C}(-,C_{i}),F\Big)\longrightarrow F(C_{i}).$$
Now, we have the isomorphism
$$\xymatrix{\Delta:\mathrm{Hom}\Big(\bigoplus_{i\in I}\mathcal{C}(-,C_{i}),F\Big)\ar[r] & \prod_{i\in I}\mathrm{Hom}\Big(\mathcal{C}(-,C_{i}),F\Big)\ar[r] & \prod_{i\in I}F(C_{i}).}$$
Then the morphism  $\varphi$ given in the diagram
$$\xymatrix{ \bigoplus_{i\in I}\mathcal{C}(-,C_{i})\ar[r]^(.7){\varphi} & F\ar[r] & 0}$$
determines an element 
$$x_{\varphi}=(x_{i})_{i\in I}\in \prod_{i\in I}F(C_{i})$$
where $x_{i}:=Y_{i}(\varphi_{i})=(\varphi_{i})_{C_{i}}(1_{C_{i}})$ with $\varphi_{i}=\varphi \circ u_{i}$ for all $i\in I$.\\
Since $\Gamma$ is surjective, there exists $y\in  \bigoplus_{\lambda \in \Lambda}F(C_{\lambda})$ such that $\Gamma(y)=x_{\varphi}$. Hence, $y=(y_{\lambda})_{\in \Lambda}$ is such that $y_{\lambda}=0$ for infinitely many $\lambda\in \Lambda$ and $y_{\lambda}\neq 0$ for a finite set. Suppose that
$y_{\lambda_{j}}\neq 0$ for $y_{\lambda_{j}}\in F(C_{\lambda_{j}})$ with $j=1,\cdots, n$. 
Then we have morphism
$$\xymatrix{\bigoplus_{j=1}^{n}F(C_{\lambda_{j}})\ar[r]^{\mu} & \bigoplus_{\lambda \in \Lambda}F(C_{\lambda})\ar[r]^{\Gamma} & \prod_{i\in I}F(C_{i})}$$
where $\mu$ is the canonical inclusion and such that there exists $y'=(y_{1},\cdots, y_{n})\in \bigoplus_{j=1}^{n}F(C_{\lambda_{j}})$ satisfying that $\mu(y')=y$. Since $\mathcal{C}$ is additive we can consider $C':=\bigoplus_{j=1}^{n}C_{\lambda_{j}}\in \mathcal{C}$, thus $F(C')\simeq \bigoplus_{j=1}^{n}F(C_{\lambda_{j}})$.\\
Consider the inclusion $u:\bigoplus_{j=1}^{n}\mathcal{C}(C_{\lambda_{j}},-)\longrightarrow  \bigoplus_{\lambda \in \Lambda}\mathcal{C}(C_{\lambda},-)$, hence we have the morphism 

$$\xymatrix{\mathcal{C}(C',-)=\bigoplus_{j=1}^{n}\mathcal{C}(C_{\lambda_{j}},-)\ar[r]^(.6){u}& \bigoplus_{\lambda \in \Lambda}\mathcal{C}(C_{\lambda},-)\ar[r]^(.5){\Theta} & \prod_{i\in I}\mathcal{C}(C_{i},-)\ar[d]^{\pi_{i}}\\
& &\mathcal{C}(C_{i},-) }$$

By Yoneda's, for each $i\in I$ there exists a morphism $\gamma_{i}:C_{i}\longrightarrow C'$ such that 
$\mathcal{C}(\gamma_{i},-)=\pi_{i}\Theta u $.
Then 
$$F(\gamma_{i})=F\otimes \mathcal{C}(\gamma_{i},-)=(F\otimes \pi_{i})\circ (F\otimes \Theta)\circ (F\otimes u).$$
Thus, since  $F$ preserve colimits and by the equality above we have the commutative diagram

$$\xymatrix{F(C')=\bigoplus_{j=1}^{n}F(C_{\lambda_{j}})\ar[r]^(.6){\mu=F\otimes u}\ar[dr]_{F(\gamma_{i})=F\otimes \mathcal{C}(\gamma_{i},-)\,\,\,\,\,\,\,\,\,\,\,} & \bigoplus_{\lambda \in \Lambda}F(C_{\lambda})\ar[r]^(.4){F\otimes\Theta} &  F\otimes \Big(\prod_{i\in I}\mathcal{C}(-,C_{i})\Big)\ar[d]^(.6){\delta}\ar[dl]^{F\otimes \pi_{i}}\\
 & F(C_{i}) &  \prod_{i\in I}F(C_{i})\ar[l]^{p_{i}} & & }$$ 
where $p_{i}$ is the $i-th$ projection (recall that $F\in \mathrm{Mod}(\mathcal{C}^{op})$, that is $F$ is contravariant). By construction, we have that $y'\in F(C')=\bigoplus_{j=1}^{n}F(C_{\lambda_{j}})$ satisfies that $(p_{i}\circ\delta\circ (F\otimes\Theta)\circ \mu)(y')=x_{i}$ and thus $F(\gamma_{i})(y')=x_{i}$.\\
By Yoneda's isomorphism 
$$Y:\mathrm{Hom}_{\mathrm{Mod}(\mathcal{C}^{op})}\Big(\mathcal{C}(-,C'),F\Big)\longrightarrow F(C')$$
there exists an element $\eta:\mathcal{C}(-,C')\longrightarrow F$ such that $Y(\eta)=y'$. We also have the following commutative diagram
$$\xymatrix{\mathrm{Hom}\Big(\mathcal{C}(-,C'),F\Big)\ar[rr]^{\mathrm{Hom}(\mathcal{C}(-,\gamma_{i}),F)}\ar[d]^{Y} & & \mathrm{Hom}\Big(\mathcal{C}(-,C_{i}),F\Big)\ar[d]^{Y_{i}}\\
F(C')\ar[rr]^{F(\gamma_{i})} & & F(C_{i})}$$
Therefore we have that 
$$Y_{i}\Big(\mathrm{Hom}\Big(\mathcal{C}(-,\gamma_{i}),F\Big)(\eta)\Big)=F(\gamma_{i})(Y(\eta))=F(\gamma_{i})(y')=x_{i}.$$
Since $Y_{i}$ is an isomorphism and $\varphi_{i}=\varphi\circ u_{i}$ satisfies that $Y_{i}(\varphi_{i})=x_{i}$, we conclude that $$\mathrm{Hom}(\mathcal{C}(-,\gamma_{i}),F)(\eta)=\varphi_{i}.$$

That is the following diagrama commutes
$$\xymatrix{\mathcal{C}(-,C_{i})\ar[r]^{\varphi_{i}}\ar[d]_{\mathcal{C}(-,\gamma_{i})} & F\\
\mathcal{C}(-,C')\ar[ur]_{\eta} & }$$

The family $\{\mathcal{C}(-,\gamma_{i}):\mathcal{C}(-,C_{i})\longrightarrow \mathcal{C}(-,C')\}_{i\in I}$ determines a morphism $\gamma:\bigoplus_{i\in I}\mathcal{C}(-,C_{i})\longrightarrow \mathcal{C}(-,C')$ such that  $\gamma\circ u_{i}=\mathcal{C}(-,\gamma_{i})$ for all $i\in I$. Then, the following diagram commutes

$$\xymatrix{\bigoplus_{i\in I}\mathcal{C}(-,C_{i})\ar[rr]^{\varphi}\ar[d]_{\gamma}  & & F\\
\mathcal{C}(-,C')\ar[urr]_{\eta} & }$$
Since $\varphi$ is an epimorphism we have that $\eta$ is an epimorphism. Hence $F$ es finitely generated.\\

$(b)\Longrightarrow (a)$. 
Consider $\prod_{i\in I}M_{i}\in \mathrm{Mod}(\mathcal{C})$ and $\pi_{i}:\prod_{i\in I}M_{i}\longrightarrow  M_{i}$ the canonical projections. Then we have a family of morphisms in $\text{Ab}$:
$$\left\{F\otimes \pi_{i}:F\otimes \Big(\prod_{i\in I}M_{i}\Big)\longrightarrow F\otimes M_{i}\right\}_{i\in I}$$
Then, there exists a unique morphism  $\delta:F\otimes (\prod_{i\in I}M_{i})\longrightarrow \prod_{i\in I}(F\otimes M_{i})$  such that the following diagram commutes
$$\xymatrix{F\otimes \Big(\prod_{i\in I}M_{i}\Big)\ar[rr]^{\delta}\ar[dr]_{F\otimes \pi_{i}} & & \prod_{i\in I}(F\otimes M_{i})\ar[dl]_{p_{i}}\\
 & F\otimes M_{i} &}$$
 where $p_{i}: \prod_{i\in I}(F\otimes M_{i})\longrightarrow F\otimes M_{i}$ are the canonical projections of the product in $\text{Ab}$. Suppose that $F$ is finitely generated. Hence there exists an epimorphism
$$\eta:\mathcal{C}(-,C)\longrightarrow F$$ Hence we have a commutative diagram
$$\xymatrix{\mathcal{C}(-,C)\otimes (\prod_{i\in I}M_{i})\ar[rr]^{\eta\otimes ( \prod_{i\in I}M_{i})}\ar[d]^{\lambda}& & F\otimes (\prod_{i\in I}M_{i})\ar[r]\ar[d]^{\delta} & 0\\
\prod_{i\in I}(F\otimes \mathcal{C}(C_{i},-))\ar[rr]_{\prod_{i\in I}(\eta \otimes M_{i})} & &  \prod_{i\in I}(F\otimes M_{i})\ar[r] & 0}$$
where $\lambda$ is an isomorphism and $\prod_{i\in I}(\eta \otimes M_{i})$ is an epimorphism since it is product of epimorphisms. We conclude that $\delta$ is an epimorphism.
\end{proof}

The following  is a generalization of  the Satz 2 in \cite{Lenzing1}.

\begin{proposition}\label{SatzLenzing2}
Let $F \in \mathrm{Mod}(\mathcal{C})$ be. Then, $F\otimes-$ preserve products if and only if $F$ is finitely presented.
\end{proposition}
\begin{proof}
Suppose that $F$ is finitely presented. Then we have exact sequence
$$\xymatrix{\mathcal{C}(X,-)\ar[r] & \mathcal{C}(Y,-)\ar[r] & F\ar[r] & 0.}$$
For $\prod_{i\in I}M_{i}$, we have the following diagram
$$\xymatrix{\mathcal{C}(X,-)\otimes (\prod_{i\in I}M_{i})\ar[r]\ar[d]^{\lambda_{1}} & \mathcal{C}(Y,-)\otimes (\prod_{i\in I}M_{i})\ar[r]\ar[d]^{\lambda_{2}} & F\otimes (\prod_{i\in I}M_{i})\ar[r]\ar[d]^{\lambda_{3}} & 0\\
\prod_{i\in I} (\mathcal{C}(X,-)\otimes M_{i})\ar[r] & \prod_{i\in I} (\mathcal{C}(X,-)\otimes M_{i})\ar[r] & \prod_{i\in I} (F\otimes M_{i})\ar[r] & 0   }$$
where $\lambda_{1}$ and $\lambda_{2}$ are isomorphisms and thus $\lambda_{3}$ is an isomorphism.\\
Suppose that $F\otimes-$ preserve products, by Proposition \ref{criteriofg}, we get that $F$ is finitely generated. Thus, we get exact sequence
$$\xymatrix{0\ar[r] & K\ar[r] & \mathcal{C}(X,-)\ar[r] & F\ar[r] & 0.}$$
Let us see that $K$ is finitely generated. Indeed, we consider $\prod_{i\in I}\mathcal{C}(-,C_{i})$, and we obtain the following commutative and exact diagram
$$\scalebox{0.8}{\xymatrix{ & K\otimes ( \prod_{i\in I}\mathcal{C}(-,C_{i}))\ar[r]\ar[d]^{\lambda_{1}} & \mathcal{C}(X,-)\otimes (\prod_{i\in I}\mathcal{C}(-,C_{i}))\ar[r]\ar[d]^{\lambda_{2}} & F\otimes (\prod_{i\in I}\mathcal{C}(-,C_{i}))\ar[r]\ar[d]^{\lambda_{3}} & 0\\ 
0\ar[r] & \prod_{i\in I} K(C_{i})\ar[r] & \prod_{i\in I} \mathcal{C}(X,C_{i})\ar[r] & \prod_{i\in I} F(C_{i})\ar[r] & 0   }}$$
Since $\lambda_{2}$ and  $\lambda_{3}$ are isomorphisms, we get that  $\lambda_{1}$ is an epimorphism (by Snake's Lemma) and thus by Proposition \ref{criteriofg}, we have that $K$ is finitely generated, proving that $F$ is finitely presented.
\end{proof}

In the following we will give a proof of the Proposition \ref{Janettheorem} which is stated in \cite[Theorem 4.4 in p. 289]{Janet} since we could not find a proof in the literature. In order to do this, first we recall the following construction.

Let $\mathcal{C}$ be a preadditive category, we recall the construction of the functor 
$(-)^{\ast}:\mathrm{Mod}(\mathcal{C})\longrightarrow \mathrm{Mod}(\mathcal{C}^{op})$
which is a generalization of the functor $\mathrm{Mod}(A)\longrightarrow \mathrm{Mod}(A^{op})$ given by $M\mapsto \mathrm{Hom}_{A}(M,A)$ for all the $A$-modules $M$.\\
Indeed, for each $\mathcal{C}$-module $M$ we define $M^{\ast}:\mathcal{C}\longrightarrow \mathbf{Ab}$ given by $M^{\ast}(C)=\mathrm{Hom}_{\mathrm{Mod}(\mathcal{C})}(M,\mathrm{Hom}_{\mathcal{C}}(C,-))$. Clearly $M^{\ast}$ is a $\mathcal{C}^{op}$-module. In this way we obtain a contravariant functor $(-)^{\ast}:\mathrm{Mod}(\mathcal{C})\longrightarrow \mathrm{Mod}(\mathcal{C}^{op})$ given by $M\mapsto M^{\ast}$.\\
If $M=\mathrm{Hom}_{\mathcal{C}}(C,-)$ it can be seen that $M^{\ast}=\mathrm{Hom}_{\mathcal{C}}(-,C)$, we refer the reader to p. 336-337 in section 6 in \cite{AusVarietyI} for more details.
We have the following result given in p. 340 in \cite{AusVarietyI}.

\begin{proposition}\label{morfismo canonico}
Let $F\in \mathrm{Mod}(\mathcal{C}^{op})$ be. There exists a canonical morphism of functors
$$\beta^{F^{\ast}}:-\otimes_{\mathcal{C}}F^{\ast}\longrightarrow \mathrm{Hom}_{\mathrm{Mod}(\mathcal{C}^{op})}(F,-).$$
In particular, for $\varphi:G\longrightarrow H$ a morphism in $\mathrm{Mod}(\mathcal{C}^{op})$, the following diagram commutes
$$\xymatrix{ G \otimes F^{\ast} \ar[rr]^(.5){\varphi\otimes F^{\ast}}\ar[d]_{\beta_{G}^{F^{\ast}}} & & H\otimes F^{\ast}\ar[d]^{\beta_{H}^{F^{\ast}}} \\
\mathrm{Hom}(F, G)\ar[rr]^{\mathrm{Hom}(F, \varphi)} & & \mathrm{Hom}(F,H)}$$
\end{proposition}
\begin{proof}
We have two functors
$$-\otimes_{\mathcal{C}}F^{\ast}:\mathrm{Mod}(\mathcal{C}^{op})\longrightarrow \mathrm{Ab}$$
$$\mathrm{Hom}_{\mathrm{Mod}(\mathcal{C}^{op})}(F,-):\mathrm{Mod}(\mathcal{C}^{op})\longrightarrow \mathrm{Ab}$$
We define $$\beta^{F^{\ast}}:-\otimes_{\mathcal{C}}F^{\ast}\longrightarrow \mathrm{Hom}_{\mathrm{Mod}(\mathcal{C}^{op})}(F,-)$$
such that for $G=\mathrm{Hom}_{\mathcal{C}}(-,C)$  we have an isomorphism
$$\beta_{G}^{F^{\ast}}:\mathrm{Hom}_{\mathcal{C}}(-,C)\otimes F^{\ast}\longrightarrow F^{\ast}(C)=\mathrm{Hom}_{\mathrm{Mod}(\mathcal{C}^{op})}(F,\mathrm{Hom}_{\mathcal{C}}(-,C)).$$
Since $\{\mathrm{Hom}_{\mathcal{C}}(-,C)\}_{C\in \mathcal{C}}$ is a set of finitely generated projectives of $\mathrm{Mod}(\mathcal{C}^{op})$, we extend this to a morphism 
$$\beta^{F^{\ast}}:-\otimes_{\mathcal{C}}F^{\ast}\longrightarrow \mathrm{Hom}_{\mathrm{Mod}(\mathcal{C}^{op})}(F,-)$$
In particular, for $\varphi:G\longrightarrow H$ a  morphism in $\mathrm{Mod}(\mathcal{C}^{op})$, the following diagram commutes
$$\xymatrix{ G \otimes F^{\ast} \ar[r]\ar[d] & H\otimes F^{\ast}\ar[d] \\
\mathrm{Hom}(F, G)\ar[r] & \mathrm{Hom}(F,H)}$$
\end{proof}

\begin{proposition}\label{unisoproj}
Let $F\in \mathrm{Mod}(\mathcal{C}^{op})$ be. 
Suppose that $\beta_{F}^{F^{\ast}}:F\otimes_{\mathcal{C}}F^{\ast}\longrightarrow \mathrm{Hom}_{\mathrm{Mod}(\mathcal{C}^{op})}(F,F)$ is an epimorphism. Then
$F$ is a finitely generated projective $\mathcal{C}^{op}$-module.
\end{proposition}
\begin{proof}
We have an epimorhism $$\xymatrix{ \bigoplus_{\lambda\in \Lambda}\mathcal{C}(-,C_{\lambda})\ar[r]^(.7){\varphi} & F\ar[r] & 0.}$$
Let $G:= \bigoplus_{\lambda\in \Lambda}\mathcal{C}(-,C_{\lambda})$ be. By Proposition \ref{morfismo canonico}, we have the following commutative diagram 
$$(\ast):\xymatrix{\Big(\bigoplus_{\lambda\in \Lambda}\mathcal{C}(-,C_{\lambda})\Big)\otimes F^{\ast} \ar[rr]^(.6){\varphi\otimes  F^{\ast}}\ar[d]_{\beta_{G}^{F^{\ast}}} &  & F\otimes F^{\ast} \ar[r]\ar[d]^{\beta_{F}^{F^{\ast}}} & 0\\
\mathrm{Hom}_{\mathrm{Mod}(\mathcal{C}^{op})}\Big(F, \bigoplus_{\lambda\in \Lambda}\mathcal{C}(-,C_{\lambda})\Big)\ar[rr]^(.6){\mathrm{Hom}(F,\varphi)} &  & \mathrm{Hom}_{\mathrm{Mod}(\mathcal{C}^{op})}(F,F)}$$
If $\beta_{F}^{F^{\ast}}$ is an epimorphism, we have that $\beta_{F}^{F^{\ast}}\circ (\varphi\otimes F^{\ast})$ is an epimorphism. Hence, there exists $y\in   \Big(\bigoplus_{\lambda\in \Lambda}\mathcal{C}(-,C_{\lambda})\Big)\otimes F^{\ast}$ such that 
$$(\beta_{F}^{F^{\ast}}\circ (\varphi\otimes F^{\ast}))(y)=1_{F}.$$
But $\Big(\bigoplus_{\lambda\in \Lambda}\mathcal{C}(-,C_{\lambda})\Big)\otimes F^{\ast}=\bigoplus_{\lambda\in \Lambda}\mathcal{C}(-,C_{\lambda})\otimes F^{\ast}=\bigoplus_{\lambda\in \Lambda}F^{\ast}(C_{\lambda})$, hence we can think that $y\in \bigoplus_{\lambda\in \Lambda}F^{\ast}(C_{\lambda})$.\\
Hence, $y=(y_{\lambda})_{\in \Lambda}$ is such that $y_{\lambda}=0$ for infinitely many $\lambda\in \Lambda$ and $y_{\lambda}\neq 0$ for a finite set. Suppose that
$y_{\lambda_{j}}\neq 0$ for $y_{\lambda_{j}}\in F^{\ast}(C_{\lambda_{j}})$ with $j=1,\cdots, n$. 
Then we have morphism
$$\xymatrix{\bigoplus_{j=1}^{n}F^{\ast}(C_{\lambda_{j}})\ar[r]^{\mu} & \bigoplus_{\lambda \in \Lambda}F^{\ast}(C_{\lambda})}$$
where $\mu$ is the canonical inclusion and such that there exists $y'=(y_{1},\cdots, y_{n})\in \bigoplus_{j=1}^{n}F^{\ast}(C_{\lambda_{j}})$ satisfying that $\mu(y')=y$.\\
Now, we have the morphism 
$$\xymatrix{\bigoplus_{j=1}^{n}\mathcal{C}(-,C_{\lambda_{j}})\ar[r]^(.5){u}& \bigoplus_{\lambda \in \Lambda}\mathcal{C}(-,C_{\lambda}).}$$
Let $H:=\bigoplus_{j=1}^{n}\mathcal{C}(-,C_{\lambda_{j}})$ be.  Since $-\otimes_{\mathcal{C}}F^{\ast}$ preserve colimits, by  Proposition \ref{morfismo canonico} we have the following commutative diagram

$$\scalebox{0.9}{\xymatrix{(\ast\ast): \Big(\bigoplus_{j=1}^{n}\mathcal{C}(-,C_{\lambda_{j}})\Big)\otimes F^{\ast}\ar[rr]^{u\otimes F^{\ast}=\mu}\ar[d]^{\beta_{H}^{F^{\ast}}} & &\Big(\bigoplus_{\lambda\in \Lambda}\mathcal{C}(-,C_{\lambda})\Big)\otimes F^{\ast}\ar[d]^{\beta_{G}^{F^{\ast}}}\\
\mathrm{Hom}_{\mathrm{Mod}(\mathcal{C}^{op})}\Big(F, \bigoplus_{j=1}^{n}\mathcal{C}(-,C_{\lambda_{j}})\Big)\ar[rr]^(.5){\mathrm{Hom}(F,u)} & &  \mathrm{Hom}_{\mathrm{Mod}(\mathcal{C}^{op})}\Big(F,\bigoplus_{\lambda\in \Lambda}\mathcal{C}(-,C_{\lambda})\Big)}}$$

By the diagrams $(\ast)$ and $(\ast\ast)$, we have that $y'\in \Big(\bigoplus_{j=1}^{n}\mathcal{C}(-,C_{\lambda_{j}})\Big)\otimes F^{\ast}$ satisfies that $(\beta_{F}^{F^{\ast}}\circ (\varphi\otimes F^{\ast})\circ \mu)(y')=1_{F}$.
Thus, $\Big(\mathrm{Hom}(F,\varphi)\circ \mathrm{Hom}(F,u)\circ \beta_{H}^{F^{\ast}}\Big)(y')=1_{F}$.\\
Therefore, $\beta_{H}^{F^{\ast}}(y')\in \mathrm{Hom}_{\mathrm{Mod}(\mathcal{C}^{op})}\Big(F, \bigoplus_{j=1}^{n}\mathcal{C}(-,C_{\lambda_{j}})\Big)$ is such that the following diagram commutes
$$\xymatrix{F\ar[rr]^{1}\ar[dr]_{\beta_{H}^{F^{\ast}}(y')} & & F\\
 & 	\bigoplus_{j=1}^{n}\mathcal{C}(-,C_{\lambda_{j}})\ar[ur]_{\varphi\circ u} & }$$
 Hence, $F$ is a direct summand of $\bigoplus_{j=1}^{n}\mathcal{C}(-,C_{\lambda_{j}})$ and we conclude that $F$ is a finitely generated projective $\mathcal{C}^{op}$-module.
\end{proof}

\begin{lemma}\label{cuadradoconmu}
Let $\eta:L\longrightarrow F$ be a morphism  in $\mathrm{Mod}(\mathcal{C}^{op})$ and $\eta^{\ast}:F^{\ast}\longrightarrow L^{\ast}$ the corresponding morphism in $\mathrm{Mod}(\mathcal{C})$. Then, we can extend
$\eta^{\ast}$ to a natural transformation 
$$-\otimes_{\mathcal{C}}\eta^{\ast}:-\otimes_{\mathcal{C}}F^{\ast}\longrightarrow -\otimes_{\mathcal{C}}L^{\ast}.$$
such that the following diagram commutes
$$\xymatrix{\otimes_{\mathcal{C}}F^{\ast}\ar[rr]^{\beta^{F^{\ast}}}\ar[d]_{-\otimes_{\mathcal{C}}\eta^{\ast}} & & \mathrm{Hom}_{\mathrm{Mod}(\mathcal{C}^{op})}(F,-)\ar[d]^{\mathrm{Hom}_{\mathrm{Mod}(\mathcal{C}^{op})}(\eta,-)}\\
\otimes_{\mathcal{C}}L^{\ast}\ar[rr]^{\beta^{L^{\ast}}} & & \mathrm{Hom}_{\mathrm{Mod}(\mathcal{C}^{op})}(L,-)}$$
\end{lemma}
\begin{proof}
Recall that for $C\in \mathcal{C}$ we have that
$$\eta^{\ast}_{C}:F^{\ast}(C)=\mathrm{Hom}_{\mathrm{Mod}(\mathcal{C}^{op})}\Big(F,\mathcal{C}(-,C)\Big)\longrightarrow L^{\ast}(C)=\mathrm{Hom}_{\mathrm{Mod}(\mathcal{C}^{op})}\Big(L,\mathcal{C}(-,C)\Big)$$
is just defined as $\eta^{\ast}_{C}:=\mathrm{Hom}_{\mathrm{Mod}(\mathcal{C}^{op})}\Big(\eta,\mathcal{C}(-,C)\Big).$
Let $\mathcal{P}:=\{\mathcal{C}(-,C)\}_{C\in \mathcal{C}}\subseteq \mathrm{Mod}(\mathcal{C}^{op})$ be.
Since $\mathcal{C}(-,C)\otimes_{\mathcal{C}}F^{\ast}\simeq F^{\ast}(C)$ and $\mathcal{C}(-,C)\otimes_{\mathcal{C}}L^{\ast}\simeq L^{\ast}(C)$ we define $\mathcal{C}(-,C)\otimes\eta^{\ast}:=\eta^{\ast}_{C}.$   In this way we have a natural transformation
$$-\otimes_{\mathcal{C}}\eta^{\ast}: (-\otimes_{\mathcal{C}}F^{\ast})|_{\mathcal{P}}\longrightarrow (-\otimes_{\mathcal{C}}L^{\ast})|_{\mathcal{P}}.$$
Since $\{\mathcal{C}(-,C)\}_{C\in \mathcal{C}}$ is a set of generators which are finitely generated and projectives and 
$-\otimes_{\mathcal{C}}F^{\ast}$ preserve colimits, by \cite[Theorem 6.1]{Mitchell}, we extend this to unique a natural transformation:
$$-\otimes_{\mathcal{C}}\eta^{\ast}:-\otimes_{\mathcal{C}}F^{\ast}\longrightarrow -\otimes_{\mathcal{C}}L^{\ast}.$$

For $G=\mathrm{Hom}_{\mathcal{C}}(-,C)$ we have the following commutative diagram

$$\xymatrix{\mathrm{Hom}_{\mathrm{Mod}(\mathcal{C}^{op})}\Big(F,\mathcal{C}(-,C)\Big)\ar[rr]^{\beta_{G}^{F^{\ast}}=1}\ar[d]_{\eta_{C}^{\ast}}& & \mathrm{Hom}_{\mathrm{Mod}(\mathcal{C}^{op})}\Big(F,\mathcal{C}(-,C)\Big)\ar[d]^{\mathrm{Hom}_{\mathrm{Mod}(\mathcal{C}^{op})}\Big(\eta,\mathcal{C}(-,C)\Big)}\\
\mathrm{Hom}_{\mathrm{Mod}(\mathcal{C}^{op})}\Big(L,\mathcal{C}(-,C)\Big)\ar[rr]_{\beta_{G}^{L^{\ast}}=1}& &\mathrm{Hom}_{\mathrm{Mod}(\mathcal{C}^{op})}\Big(L,\mathcal{C}(-,C)\Big)}$$
Since, $\{\mathcal{C}(-,C)\}_{C\in \mathcal{C}}$ is a set of generators which are finitely generated and projectives of $\mathrm{Mod}(\mathcal{C}^{op})$ and 
$-\otimes_{\mathcal{C}}F^{\ast}$ preserve colimits, by \cite[Theorem 6.1]{Mitchell},  we have that  the following diagram commutes
$$\xymatrix{\otimes_{\mathcal{C}}F^{\ast}\ar[rr]^{\beta^{F^{\ast}}}\ar[d]_{-\otimes_{\mathcal{C}}\eta^{\ast}} & & \mathrm{Hom}_{\mathrm{Mod}(\mathcal{C}^{op})}(F,-)\ar[d]^{\mathrm{Hom}_{\mathrm{Mod}(\mathcal{C}^{op})}(\eta,-)}\\
\otimes_{\mathcal{C}}L^{\ast}\ar[rr]^{\beta^{L^{\ast}}} & & \mathrm{Hom}_{\mathrm{Mod}(\mathcal{C}^{op})}(L,-).}$$
\end{proof}

We will give a proof of the following result since we could not find a prove in the literature.

\begin{proposition}$\textnormal{\cite[Theorem 4.4 in p. 289]{Janet}}$\label{Janettheorem}
Let $F\in \mathrm{Mod}(\mathcal{C}^{op})$. The following are equivalent
\begin{enumerate}
\item [(a)] $F\otimes_{\mathcal{C}}-:\mathrm{Mod}(\mathcal{C})\longrightarrow \mathrm{Ab}$  preserve inverse limits.

\item [(b)] $\alpha_{G}:F\otimes G^{\ast} \longrightarrow \mathrm{Hom}(G,F)$ is an isomorphism for all $G\in \mathrm{Mod}(\mathcal{C}^{op})$.

\item [(c)] $\alpha_{G}:F\otimes G^{\ast}\longrightarrow \mathrm{Hom}(G,F)$ is an epimorphism for all $G\in \mathrm{Mod}(\mathcal{C}^{op})$.

\item [(d)] $\alpha_{F}:F\otimes F^{\ast}\longrightarrow \mathrm{Hom}(F,F)$ is an epimorphism.

\item [(e)] $F$ is a finitely generated projective $\mathcal{C}^{op}$-module.
\end{enumerate}

\end{proposition}
\begin{proof}
$(a)\Rightarrow (b)$. We consider the contravariant functor
$$(-)^{\ast}:\mathrm{Mod}(\mathcal{C}^{op})\longrightarrow \mathrm{Mod}(\mathcal{C}).$$
Then, we have contravariant functors 
$$(F\otimes_{\mathcal{C}}-)\circ (-)^{\ast}:\mathrm{Mod}(\mathcal{C}^{op})\longrightarrow \mathrm{Ab}$$
$$\mathrm{Hom}_{\mathrm{Mod}(\mathcal{C}^{op})}(-,F):\mathrm{Mod}(\mathcal{C}^{op})\longrightarrow \mathrm{Ab}$$
We will construct an isomorphism
$$\alpha:(F\otimes_{\mathcal{C}}-)\circ (-)^{\ast} \longrightarrow \mathrm{Hom}_{\mathrm{Mod}(\mathcal{C}^{op})}(-,F).$$
For this, we consider $G\in \mathrm{Mod}(\mathcal{C}^{op})$. We have exact sequence in $\mathrm{Mod}(\mathcal{C}^{op})$
$$\xymatrix{H\ar[r]^{\eta} & L\ar[r]^(.6){\varphi} & G\ar[r] & 0.}$$
with $L:=\bigoplus_{j\in J}\mathcal{C}(-,C_{j})$ and $H:=\bigoplus_{i\in I}\mathcal{C}(-,C_{i})$.
By applying  the funtor $(-)^{\ast}$, we have the exact sequence in $\mathrm{Mod}(\mathcal{C})$:
$$\xymatrix{0\ar[r] & G^{\ast}\ar[r]^{\varphi^{\ast}}  & L^{\ast}\ar[r]^{\eta^{\ast}} & H^{\ast}}$$
If $F\otimes_{\mathcal{C}}-$ preserve inverse limits in particular it preserve kernels, then we have the following exact sequence
$$\xymatrix{0\ar[r] & F\otimes G^{\ast}\ar[r]^{F\otimes \varphi^{\ast}}  & F\otimes L^{\ast}\ar[r]^{F\otimes \eta^{\ast}}  & F\otimes H^{\ast}}$$
Since $\mathrm{Hom}(-,F)$ is left exact we have exact sequence
$$\xymatrix{0\ar[r] & \mathrm{Hom}(G,F)\ar[r]^(.5){\mathrm{Hom}(\varphi,F)} & \mathrm{Hom}(L,F)\ar[r]^{\mathrm{Hom}(\eta,F)}  &  \mathrm{Hom}(H,F)}$$
On the other hand, by Lemma \ref{cuadradoconmu}, we have the following commutative diagram
$$\xymatrix{-\otimes G^{\ast}\ar[rr]^{-\otimes \varphi^{\ast}}\ar[d]^{\beta^{G^{\ast}}}  && -\otimes L^{\ast}\ar[rr]^{-\otimes \eta^{\ast}}\ar[d]^{\beta^{L^{\ast}}}  && -\otimes H^{\ast}\ar[d]^{\beta^{H^{\ast}}}\\
\mathrm{Hom}(G,-)\ar[rr]^{\mathrm{Hom}(\varphi,-)} & & \mathrm{Hom}(L,-)\ar[rr]^{\mathrm{Hom}(\eta,-)}  & &\mathrm{Hom}(H,-)}$$
Hence evaluating in $F$ we have the following commutative diagram

$$\xymatrix{0\ar[r] & F\otimes G^{\ast}\ar[r]^(.4){F\otimes \varphi^{\ast}}\ar[d]^{\beta^{G^{\ast}}_{F}}  & F\otimes \Big(\prod_{j\in J}\mathcal{C}(C_{j},-) \Big)\ar[r]^{F\otimes \eta^{\ast}}\ar[d]^{\beta^{L^{\ast}}_{F}}  & F\otimes \Big(\prod_{i\in I}\mathcal{C}(C_{i},-)\Big)\ar[d]^{\beta^{H^{\ast}}_{F}}\\
0\ar[r] & \mathrm{Hom}(G,F)\ar[r]^(.4){\mathrm{Hom}(\varphi,F)} & \mathrm{Hom}(\bigoplus_{j\in J}\mathcal{C}(-,C_{j}),F)\ar[r]^{\mathrm{Hom}(\eta,F)}  &  \mathrm{Hom}(\bigoplus_{i\in I}\mathcal{C}(-,C_{i}),F) }$$
since $L^{\ast}=\prod_{j\in J}\mathcal{C}(C_{j},-)$ and  $H^{\ast}=\prod_{i\in I}\mathcal{C}(C_{i},-)$.\\
Now, since $F\otimes -$ preserve inverse limits, in particular preserve products  and hence
$$F\otimes \Big(\prod_{j\in J}\mathcal{C}(C_{j},-)\Big)\simeq \prod_{j\in J} (F\otimes \mathcal{C}(C_{j},-))\simeq \prod_{j\in J} F(C_{j}).$$
We also have that
$$ \mathrm{Hom}\Big(\bigoplus_{j\in J}\mathcal{C}(-,C_{j}),F\Big)\simeq \prod_{j\in J}\mathrm{Hom}(\mathcal{C}(-,C_{j}),F)\simeq \prod_{j\in J} F(C_{j})$$
We conclude that $\beta^{H^{\ast}}_{F}$ and $\beta^{L^{\ast}}_{F}$ are isomorphisms.
Hence, we have isomorphism $\beta^{G^{\ast}}_{F}:F\otimes G^{\ast}\longrightarrow \mathrm{Hom}(G,F)$.\\
In this way, we define isomorphism
$$\alpha:(F\otimes_{\mathcal{C}}-)\circ (-)^{\ast} \longrightarrow \mathrm{Hom}_{\mathrm{Mod}(\mathcal{C}^{op})}(-,F).$$ as $\alpha_{G}:=\beta^{G^{\ast}}_{F}$ for all $G\in \mathrm{Mod}(\mathcal{C}^{op})$. We have proved that
$(a)\Rightarrow (b)$.\\
The implications $(b)\Rightarrow (c)$, $(c)\Rightarrow (d)$ are trivial.\\
The implication $(d)\Rightarrow (e)$ is just the Proposition \ref{unisoproj}.
Finally, the implication $(e)\Rightarrow (a)$ is easy, since for $F=\mathrm{Hom}_{\mathcal{C}}(-,C)$ is trivial and hence for a direct summand of a representable.
\end{proof}

\begin{Remark}
We have that $\beta^{G^{\ast}}_{F}:F\otimes G^{\ast}\longrightarrow \mathrm{Hom}(G,F)$ is such that its image consist of the natural transformations from $G$ to $F$ that factor through a representable functor (see the last rows in p. 288 in the article \cite{Janet}).
\end{Remark}

\begin{proposition}\label{plano=proy}
Let $F\in \mathrm{Mod}(\mathcal{C})$ a finitely presented flat module.Then $F$ is a finitely generated projective module.
\end{proposition}
\begin{proof}
Since $F$ is flat, we have that $F\otimes_{\mathcal{C}}-$ is exact. Now, since $F$ is finitely presented we get that $F$ preserve arbitrary products by  Proposition \ref{SatzLenzing2}.\\
We know, that inverse limits are computed through kernels and products. Hence we conclude that $F$ preserve inverse limits. Therefore, by Theorem \ref{Janettheorem}, we conclude that   $F$ is a finitely generated projective module.
\end{proof}


\section*{Acknowledgements}
{The authors would like to thank the SECIHTI Project CBF-2023-2024-2630, and the Research Support Office of the Metropolitan Autonomous University (DAI UAM) for the support granted. }

{\footnotesize

}

{\footnotesize

\vskip3mm \noindent {\bf Martha Lizbeth Shaid Sandoval Miranda}:\\ 
{\it Departamento de Matem\'aticas, Universidad Aut\'onoma Metropolitana Unidad Iztapalapa\\
Av. San Rafael Atlixco 186, Col. Vicentina Iztapalapa 09340, M\'exico, Ciudad de M\'exico.\\ {\tt marlisha@xanum.uam.mx, marlisha@ciencias.unam.mx}}

\vskip3mm \noindent {\bf Valente Santiago Vargas}:\\ 
{\it Departamento de Matem\'aticas, Facultad de Ciencias, Universidad Nacional Aut\'onoma de M\'exico\\
Circuito Exterior, Ciudad Universitaria
C.P. 04510, Ciudad de M\'exico, MEXICO.\\ {\tt valente.santiago@ciencias.unam.mx}}

\vskip3mm \noindent  {\bf Edgar Omar Velasco P\'aez}:\\ 
{\it Facultad de Ciencias, Universidad Nacional Aut\'onoma de M\'exico\\
Circuito Exterior, Ciudad Universitaria,
C.P. 04510, Ciudad de M\'exico, MEXICO.\\ {\tt edgar\underline{\,\,}bkz13@ciencias.unam.mx }

 \noindent Departamento de Matem\'aticas, Universidad Aut\'onoma Metropolitana Unidad Iztapalapa\\
Av. San Rafael Atlixco 186, Col. Vicentina Iztapalapa 09340, M\'exico, Ciudad de M\'exico.
\\ {\tt edgar\underline{\,\,}bkz13@xanum.uam.mx }} }

\end{document}